\documentclass[journal, twocolumn]{IEEEtran}

\usepackage{xspace,amsmath,amssymb,amsthm,amsfonts,epsfig,syntonly,dsfont}
\usepackage{cite,bm,color,url,textcomp}
\usepackage{algorithmicx,algorithm,algpseudocode}
\usepackage{epstopdf,cases}
\usepackage{empheq,float}
\usepackage{graphicx,subfigure,balance}

\newtheorem{proposition}{Proposition}

\newtheorem{lemma}{Lemma}

\newtheorem{theorem}{Theorem}

\theoremstyle{definition}\newtheorem{remark}{Remark}

\long\def\symbolfootnote[#1]#2{\begingroup
\def\thefootnote{\fnsymbol{footnote}}
\footnote[#1]{#2}\endgroup}

\psfull

\allowdisplaybreaks[2]

\begin{document}

\title{Stochastic Averaging for Constrained Optimization
with Application to Online Resource Allocation}
\author{ Tianyi Chen, Aryan Mokhtari, Xin Wang, Alejandro Ribeiro, and Georgios B. Giannakis
\thanks {Work in this paper was supported by NSF 1509040, 1508993, 1423316, 1514056, 1500713, 1509005, 0952867 and ONR N00014-12-1-0997; National Natural Science Foundation of China grant 61671154.
Part of the results in this paper were presented at IEEE Global Conference on Signal Information Processing, Washington, DC, USA, December 7-9, 2016 \cite{chen2016b}.}
\thanks{T. Chen and G. B. Giannakis are with the Department of Electrical and Computer Engineering and the Digital Technology Center, University of Minnesota, Minneapolis, MN 55455 USA. Emails: \{chen3827, georgios\}@umn.edu

A. Mokhtari and A. Ribeiro are with the Department of Electrical and Systems Engineering, University of Pennsylvania, Philadelphia, PA 19104 USA. Emails: \{aryanm, aribeiro\}@seas.upenn.edu

X. Wang is with the Key Laboratory for Information Science of Electromagnetic Waves (MoE), the Department of Communication Science and Engineering, Fudan University, Shanghai, China, and also with the Department of Computer and Electrical Engineering and Computer Science, Florida Atlantic University, Boca Raton, FL 33431 USA. Email:~xwang11@fudan.edu.cn
}
}

\markboth{}{}
\maketitle

\begin{abstract}
Existing resource allocation approaches for nowadays stochastic networks are challenged to meet fast convergence and tolerable delay requirements.
The present paper leverages online learning advances to facilitate online resource allocation tasks. By recognizing the central role of Lagrange multipliers, the underlying constrained optimization problem is formulated as a machine learning task involving both training and operational modes, with the goal of learning the sought multipliers in a fast and efficient manner.
To this end, an order-optimal offline learning approach is developed first for batch training, and it is then generalized to the online setting with a procedure termed learn-and-adapt. The novel resource allocation protocol permeates benefits of \textit{stochastic approximation} and \textit{statistical learning} to obtain low-complexity online updates with learning errors close to the statistical accuracy limits, while still preserving adaptation performance, which in the stochastic network optimization context guarantees queue stability. Analysis and simulated tests demonstrate that the proposed data-driven approach improves the delay and convergence performance of existing resource allocation schemes.
\end{abstract}



\section{Introduction}\label{S:Intro}
Contemporary cloud data centers (DCs) are proliferating to provide major Internet services such as  video distribution, and data backup \cite{Data2015}.
To fulfill the goal of reducing electricity cost and improving sustainability, advanced smart grid features  are nowadays adopted by cloud networks \cite{gao2012}.
Due to their inherent stochasticity, these features challenge the overall network design, and particularly the network's resource allocation task.
Recent approaches on this topic mitigate the spatio-temporal uncertainty associated with energy prices and renewable availability \cite{guo14,Urg11,chen2016,Yao12data,chen2016jsac}. Interestingly, the resultant algorithms reconfirm the critical role of the dual decomposition framework, and its renewed popularity for optimizing modern stochastic networks.

However, slow convergence and the associated network delay of existing resource allocation schemes have recently motivated improved first- and second-order optimization algorithms \cite{liu2016,li2015,zargham2013,ribeiro2010}.
Yet, historical data have not been exploited to mitigate future uncertainty.
This is the key novelty of the present paper, which permeates benefits of statistical learning to stochastic resource allocation tasks.

Targeting this goal, a critical observation is that renowned network optimization algorithms (e.g., back-pressure and max-weight) are intimately linked with Lagrange's dual theory, where the associated multipliers admit pertinent price interpretations \cite{huang2011,huang2013,valls2015}.
We contend that learning these multipliers can benefit significantly from historical relationships and trends present in massive datasets \cite{vapnik2013}.
In this context, we revisit the stochastic network optimization problem from a machine learning vantage point, with the goal of learning the Lagrange multipliers in a fast and efficient manner.
Works in this direction include \cite{huang2014} and \cite{zhang2016}, but the methods there are more suitable for problems with \textit{two features}: 1) the network states belong to a distribution with \textit{finite} support; and, 2) the feasible set is discrete with a \textit{finite} number of actions.
Without these two features, the involved learning procedure may become intractable, and the advantageous performance guarantee may not hold.
Overall, online resource allocation capitalizing on \textit{data-driven} learning schemes remains a largely uncharted territory.


Motivated by recent advances in machine learning, we systematically formulate the resource allocation problem as an online task with batch training and operational learning-and-adapting phases. In the batch training phase, we view the empirical dual problem of maximizing the sum of finite concave functions, as an empirical risk maximization (ERM) task, which is well-studied in machine learning \cite{vapnik2013}. Leveraging our specific ERM problem structure, we modify the recently developed stochastic average gradient approach (SAGA) to fit our training setup.
SAGA belongs to the class of fast incremental gradient (a.k.a. stochastic variance reduction) methods \cite{defazio2014}, which combine the merits of both stochastic gradient, and batch gradient methods. In the resource allocation setup, our \textit{offline} SAGA yields empirical Lagrange multipliers with order-optimal linear convergence rate as batch gradient approach, and per-iteration complexity as low as stochastic gradient approach.
Broadening the static learning setup in \cite{defazio2014} and \cite{Daneshmand16}, we further introduce a dynamic resource allocation approach (that we term \textit{online} SAGA) that operates in a learn-and-adapt fashion. Online SAGA fuses the benefits of stochastic approximation and statistical learning: In the learning mode, it preserves the simplicity of offline SAGA to dynamically learn from streaming data thus lowering the training error; while it also adapts by incorporating attributes of the well-appreciated stochastic dual (sub)gradient approach (SDG) \cite{guo14,Urg11,chen2016,Yao12data,neely2010}, in order to track queue variations and thus  guarantee long-term queue stability.

In a nutshell, the main contributions of this paper can be summarized as follows.
\begin{enumerate}
\item [c1)]
Using stochastic network management as a motivating application domain, we take a fresh look at dual solvers of constrained optimization problems as machine learning iterations involving training and operational phases.
\item [c2)]
During the training phase, we considerably broaden SAGA to efficiently compute Lagrange multiplier iterates at order-optimal linear convergence, and computational cost comparable to the stochastic gradient approach.
\item [c3)]
In the operational phase, our online SAGA learns-and-adapts at low complexity from streaming data.
For allocating stochastic network resources, this leads to a cost-delay tradeoff $[\mu,(1/\!{\sqrt{\mu}})\!\log^2(\mu)]$ with high probability, which markedly improves SDG's $[\mu,{1}/{\mu}]$ tradeoff \cite{neely2010}.
\end{enumerate}

\emph{Outline}. The rest of the paper is organized as follows. The system models are described in Section~\ref{S:ModelPrelim}. The motivating resource allocation setup is formulated in Section~\ref{sec.Stoc-net}.
Section~\ref{sec.SAGA} deals with our learn-and-adapt dual solvers for constrained optimization. Convergence analysis of the novel online SAGA is carried out in Section~\ref{sec:analysis}. Numerical tests are provided in Section~\ref{sec.Num}, followed by concluding remarks in Section~\ref{sec.Cons}.

\emph{Notation}. $\mathbb{E}~(\mathbb{P})$ denotes expectation (probability); $\mathbf{1}$ denotes the all-one vector; and $\|\mathbf{x}\|$ denotes the $\ell_2$-norm of vector $\mathbf{x}$. Inequalities for vectors $\mathbf{x} > \mathbf{0}$, are defined entry wise;
$[a]^+:=\max\{a,0\}$; and $(\cdot)^{\top}$ stands for transposition.
${\cal O}(\mu)$ denotes big order of $\mu$, i.e., ${\cal O}(\mu)/\mu\rightarrow 1$ as $\mu\rightarrow 0$; and $\mathbf{o}(\mu)$ denotes small order of $\mu$, i.e., $\mathbf{o}(\mu)/\mu\rightarrow 0$ as $\mu\rightarrow 0$.

\section{Network Modeling Preliminaries}\label{S:ModelPrelim}
This section focuses on resource allocation over a sustainable DC network with ${\cal J}:=\{1,2,\ldots,J\}$ mapping nodes (MNs), and ${\cal I}:=\{1,2,\ldots,I\}$ DCs.
MNs here can be authoritative DNS servers as used by Akamai and most content delivery networks, or HTTP ingress proxies as used by Google and Yahoo! \cite{xu2013}, which collect user requests over a geographical area (e.g., a city or a state), and then forward workloads to one or more DCs distributed across a large area (e.g., a country).

Notice though, that the algorithms and their performance analysis in Sections \ref{sec:analysis} and \ref{sec.SAGA} can be applied to general resource allocation tasks, such as energy management in power systems \cite{sun2016}, cross-layer rate-power allocation in communication links \cite{gatsis2010}, and traffic control in transportation networks \cite{gregoire2015}.


\noindent\textbf{Network constraints}. Suppose that interactive workloads are allocated as in e.g., \cite{xu2013}, and only delay-tolerant workloads that are deferrable are to be scheduled across slots.
Typical examples include system updates and data backup, which provide ample optimization opportunities for workload allocation based on the dynamic variation of energy prices, and the random availability of renewable energy sources.

Suppose that the time is indexed in discrete slots $t$, and let ${\cal T}:= \{0,1, \ldots\}$ denote an infinite time horizon. With reference to Fig. \ref{fig:system}, let $v_{j,t}$ denote the amount of delay-tolerant workload arriving at MN $j$ in slot $t$, and $\tilde{\mathbf{x}}_{j,t}:=[\tilde{x}_{1,j,t},\ldots,\tilde{x}_{I,j,t}]^{\top}$ the $I\times 1$ vector collecting the workloads routed from MN $j$ to all DCs in slot $t$. With the fraction of unserved workload  buffered in corresponding queues, the queue length at each mapping node $j$ at the beginning of slot $t$, obeys the recursion
\begin{equation}\label{W-delay}
q_{j,t+1}^{\rm mn}=\Big[q_{j,t}^{\rm mn}+v_{j,t}- \textstyle\sum_{i\in {\cal I}_j}\tilde{x}_{i,j,t}\Big]^{+}, ~~~\forall j, t
\end{equation}
where ${\cal I}_j$ denotes the set of DCs that MN $j$ is connected to.

At the DC side, $x_{i,t}$ denotes the workload processed by DC $i$ during slot $t$. Unserved portions of the workloads are buffered at DC queues, whose length obeys (cf. \eqref{W-delay})
\begin{equation}\label{Q-delay}
q_{i,t+1}^{\rm dc}=\Big[q_{i,t}^{\rm dc}-x_{i,t}+ \textstyle\sum_{j\in {\cal J}_i}\tilde{x}_{i,j,t}\Big]^{+}, ~~~\forall i, t
\end{equation}
where $q_{i,t}^{\rm dc}$ is the queue length in DC $i$ at the beginning of slot $t$, and ${\cal J}_i$ denotes the set of MNs that DC $i$ is linked with.
The per-slot workload is bounded by the capacity $D_i$ for each DC $i$; that is,
\begin{equation}\label{eq.ITcap}
	0\leq x_{i,t}\leq D_i,~\forall i,t.
\end{equation}

Since the MN-to-DC link has bounded bandwidth $B_{i,j}$, it holds that
\begin{equation}\label{BW}
0\leq\tilde{x}_{i,j,t} \leq B_{i,j},~~~\forall i,j,t.
\end{equation}
If MN $j$ and DC $i$ are not connected, then $B_{i,j}\!=\!0$.

%
\begin{figure}[t]
\centering
\vspace{-0.2cm}
\includegraphics[width=0.5\textwidth]{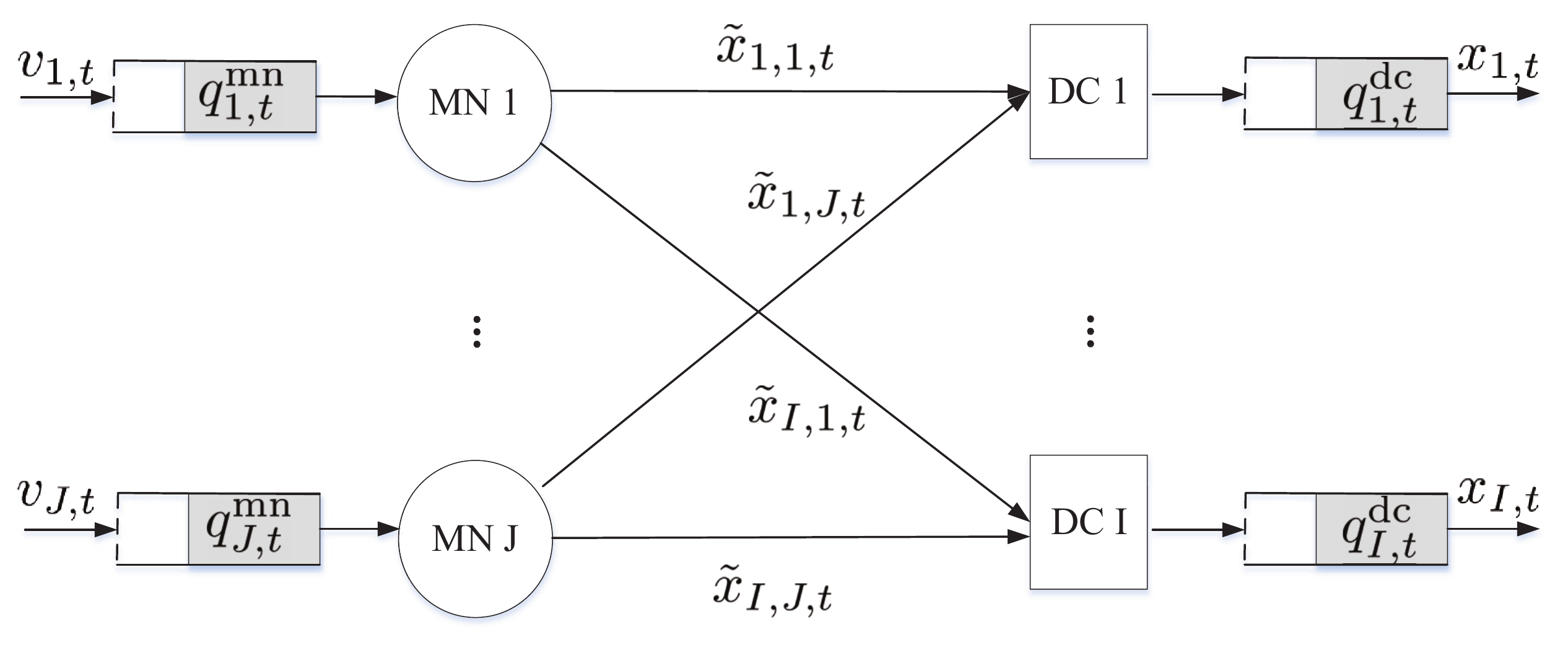}
\vspace{-0.8cm}
\caption{Diagram depicting components and variables of the MN-DC network.}
\vspace{-0.4cm}
\label{fig:system}
\end{figure}



\noindent\textbf{Operation costs}.
We will use the strictly convex function $G_{i,j}^{\rm d}(\,\cdot\,)$ to denote the cost for distributing workloads from MN $j$ to DC $i$, which depends on the distance between them.

Let $P_{i,t}^{\rm r}$ denote the energy provided at the beginning of slot $t$ by the renewable generator at DC $i$, which is constrained to lie in $[0,\overline{P}_i^{\rm r}]$.
The power consumption $P_i^{\rm dc}$ by DC $i$ is taken to be a quadratic function of the demand \cite{chen2016}; that is $P_i^{\rm dc}(x_{i,t})=e_{i,t} x_{i,t}^2$,
where $e_{i,t}$ is a time-varying parameter capturing environmental factors, such as humidity and temperature \cite{chen2016jsac}.
The energy transaction cost is modeled as a linear function of the power imbalance amount $|P_i^{\rm dc}(x_{i,t})-P_{i,t}^{\rm r}|$ to capture the cost in real time power balancing; that is,
$G_{i,t}^{\rm e}(x_{i,t}):=\alpha_{i,t}\left(e_{i,t}x_{i,t}^2 -P_{i,t}^{\rm r}\right)$, where $\alpha_{i,t}$ is the buy/sell price in the local power wholesale market.\footnote{When the prices of buying and selling are not the same, the transaction cost is a piece-wise linear function of the power imbalance amount as in \cite{chen2016}.}
Clearly, each DC should buy energy from external energy markets in slot $t$ at price $\alpha_{i,t}$ if $P_i^{\rm dc}(x_{i,t})-P_{i,t}^{\rm r}> 0$, or, sell energy to the markets with the same price, if $P_i^{\rm dc}(x_{i,t})-P_{i,t}^{\rm r}<0$.

The \emph{aggregate cost} for the considered MN-DC network per slot $t$ is a convex function, given by
\begin{align}\label{eq.net}
\Psi_t(\{x_{i,t}\},\{\tilde{\mathbf{x}}_{j,t}\})\!:=\!\sum_{i\in{\cal I}}\!G_{i,t}^{\rm e}(x_{i,t})\!+\!\sum_{i\in{\cal I}}\sum_{j\in {\cal J}} G_{i,j}^{\rm d}\Big(\tilde{\mathbf{x}}_{j,t}\Big).
\end{align}
When \eqref{eq.net} is assumed under specific operating conditions, our analysis applies to general smooth and strong-convex costs.

Building on \eqref{W-delay}-\eqref{eq.net}, the next section will deal with allocation of resources $\{x_{i,t},\tilde{\mathbf{x}}_{j,t}\}$ per slot $t$.
The goal will be to minimize the long-term average cost, subject to operational constraints.

\section{Motivating Application: Resource Allocation}\label{sec.Stoc-net}

For notational brevity, collect all random variables at time $t$ in the $(2I+J)\times 1$ state vector $\mathbf{s}_t:=\{\alpha_{i,t},e_{i,t},P_{i,t}^{\rm r},v_{j,t},\forall i,j\}$; the optimization variables in the $(IJ+I)\times 1$ vector $\mathbf{x}_t:=[\tilde{\mathbf{x}}_{1,t}^{\top},\ldots,\tilde{\mathbf{x}}_{J,t}^{\top}, x_{1,t},\ldots,x_{I,t}]^{\top}$ and let $\mathbf{q}_t:=\{q_{j,t}^{\rm mn},q_{i,t}^{\rm dc},~\forall i,j\}\in \mathbb{R}^{I+J}$.
Considering the scheduled $x_{i,t}$ as the outgoing amount of workload from DC $i$, define the $(I+J)\times (IJ+I)$ ``node-incidence" matrix with $(i,e)$ entry given by
   \begin{equation*}
	\mathbf{A}_{(i,e)}:=
	\left\{
	\begin{array}{rl}
         {1,}~  &\text{if link $e$ enters node $i$}\\
         {-1,}~ &\text{if link $e$ leaves node $i$}\\
         {0,}~ &\text{else.}
    \end{array}
   \right.
\end{equation*}
We assume that each row of
$\mathbf{A}$ has at least one $-1$ entry, and each column of
$\mathbf{A}$ has at most one $-1$ entry, meaning that each
node has at least one outgoing link, and each link has at most one
source node.
Further, collect the instantaneous workloads across MNs in the zero-padded $(I+J)\times 1$ vector $\mathbf{c}_t:=[v_{1,t},\ldots, v_{J,t}, 0, \ldots, 0]^{\top}$, as well as the link and DC capacities in the  $(I+J)\times 1$ vector $\mathbf{\bar{x}}:=[B_{1,1},\ldots, B_{i,j},D_1, \ldots, D_I]^{\top}$.

Hence, the sought scheduling is the solution of the following \textit{long-term} network-optimization problem
\begin{subequations}
\label{eq.prob}
\begin{align}
{\Psi}^{*}:=&\min_{\{\mathbf{x}_t,~\forall t\}}\, \lim_{T\rightarrow \infty}\frac{1}{T}\sum_{t=1}^T \mathbb{E}\left[\Psi_t(\mathbf{x}_t)\right]  \label{eq.proba}\\
\text{s. t. }~~~ &\mathbf{0}\leq\mathbf{x}_t \leq \mathbf{\bar{x}},~\forall t \label{eq.probl}\\
&\mathbf{q}_{t+1}=\left[\mathbf{q}_t+\mathbf{A}\mathbf{x}_t+\mathbf{c}_t\right]^{+}\!,~\forall t\label{eq.probm}\\
&\lim_{T\rightarrow \infty} \frac{1}{T} {\textstyle\sum_{t=1}^{T}} \mathbb{E}\left[\mathbf{q}_t\right]< \infty \label{eq.probn}
\end{align}
\end{subequations}
where \eqref{eq.probl} and \eqref{eq.probm} come from concatenating \eqref{eq.ITcap}-\eqref{BW} and \eqref{W-delay}-\eqref{Q-delay}, while \eqref{eq.probn} ensures strong queue stability as in \cite[Definition 2.7]{neely2010}.
The objective in \eqref{eq.proba} considers the entire time horizon, and the expectation is over all sources of randomness, namely the random vector $\mathbf{s}_t$, and the randomness of the variables $\mathbf{x}_t$ and $\mathbf{q}_t$ induced by the sample path of $\mathbf{s}_1, \mathbf{s}_2, \ldots, \mathbf{s}_t$.
For the problem \eqref{eq.prob}, the queue dynamics in \eqref{eq.probm} couple the optimization variables $\mathbf{x}_t$ over the infinite time horizon. For practical cases where the knowledge of $\mathbf{s}_t$ is causal, finding the optimal solution is generally intractable. Our approach to circumventing this obstacle is to replace \eqref{eq.probm} with limiting average constraints, and employ dual decomposition to separate the solution across time, as  elaborated next.

\subsection{Problem reformulation}\label{subsec.reform}
Substituting \eqref{eq.probm} into \eqref{eq.probn}, we will argue that the long-term aggregate (endogenous plus exogenous) workload must satisfy the following necessary condition
\begin{equation}\label{Queue-relax}
\lim_{T\rightarrow \infty}\!\frac{1}{T}\!\sum_{t=1}^T\mathbb{E}\left[\mathbf{A}\mathbf{x}_t+\mathbf{c}_t\right]\leq \mathbf{0}.
\end{equation}
Indeed, \eqref{eq.probm} implies
$\mathbf{q}_{t+1} \geq \mathbf{q}_t+\mathbf{A}\mathbf{x}_t+\mathbf{c}_t$ that after summing over
$t=1,\ldots,T$ and taking expectations yields $\mathbb{E}[\mathbf{q}_{T+1}]\geq\mathbb{E}[\mathbf{q}_1] +
\sum_{t=1}^{T} \mathbb{E}[\mathbf{A}\mathbf{x}_t+\mathbf{c}_t]$.
Since both $\mathbf{q}_1$ and $\mathbf{q}_{T+1}$ are bounded under \eqref{eq.probn},
dividing both sides by $T$ and taking $T\rightarrow
\infty$, yields \eqref{Queue-relax}.
Using \eqref{Queue-relax}, we can write a relaxed version of \eqref{eq.prob} as
%
\begin{align} \label{eq.reform}
\tilde{\Psi}^{*} := \min_{\{\mathbf{x}_t,\forall t\}} \, \lim_{T\rightarrow \infty}\frac{1}{T}\sum_{t=1}^T \mathbb{E}\left[\Psi_t(\mathbf{x}_t)\right] \quad \text{s.t.}~ \eqref{eq.probl}~{\rm and}~\eqref{Queue-relax}.
\end{align}
%

Compared to \eqref{eq.prob}, the queue variables $\mathbf{q}_t$ are not present in \eqref{eq.reform}, while the time-coupling constraints \eqref{eq.probm} and \eqref{eq.probn} are replaced with \eqref{Queue-relax}.
As \eqref{eq.reform} is a relaxed version of \eqref{eq.prob}, it follows that $\tilde{\Psi}^{*}\leq {\Psi}^{*}$.
Hence, if one solves \eqref{eq.reform} instead of \eqref{eq.prob}, it will be prudent to derive an optimality bound on ${\Psi}^{*}$, provided that the schedule obtained by solving \eqref{eq.reform} is feasible for \eqref{eq.prob}.
In addition, using arguments similar to those in \cite{neely2010} and \cite{Urg11}, it can be shown that if the random process $\{\mathbf{s}_t\}$ is stationary, there exists a stationary control policy $\bm{\chi}(\cdot)$, which is only a function of the current $\mathbf{s}_t$; it satisfies \eqref{eq.probl} almost surely; and guarantees that  $\mathbb{E}[\Psi_t(\bm{\chi}(\mathbf{s}_t))] = \tilde{\Psi}^{*}$, and $\mathbb{E}[\mathbf{A}\bm{\chi}(\mathbf{s}_t)+\mathbf{c}_t(\mathbf{s}_t)]\leq \mathbf{0}$. This implies that the infinite time horizon problem \eqref{eq.reform} is equivalent to the following per slot convex program
\begin{subequations}\label{eq.reform2}
\begin{align}
\tilde{\Psi}^{*} := &\min_{\bm{\chi}(\cdot)} \; \mathbb{E}\left[\Psi\big(\bm{\chi}(\mathbf{s}_t);\mathbf{s}_t\big)\right] \label{eq.reforme0}\\
\text{s.t.}~~~&\mathbb{E}[\mathbf{A}\bm{\chi}(\mathbf{s}_t)+\mathbf{c}_t(\mathbf{s}_t)]\leq \mathbf{0}\label{eq.reforme1}\\
&\bm{\chi}(\mathbf{s}_t)\in{\cal X}:=\{\mathbf{0}\leq \bm{\chi}(\mathbf{s}_t) \leq \mathbf{\bar{x}}\},~\forall \mathbf{s}_t\label{eq.reforme2}
\end{align}
\end{subequations}
where we interchangeably used $\bm{\chi}(\mathbf{s}_t)\!:=\!\mathbf{x}_t$ and $\Psi(\mathbf{x}_t;\mathbf{s}_t)\!:=\!\Psi_t(\mathbf{x}_t)$, to emphasize the dependence of the real-time cost $\Psi_t$ and decision $\mathbf{x}_t$ on the random state $\mathbf{s}_t$.
Note that the optimization in \eqref{eq.reform2} is w.r.t. the stationary policy $\bm{\chi}(\cdot)$. Hence, there is an infinite number of variables in the primal domain. Observe though, that there is a finite number of constraints coupling the realizations (cf. \eqref{eq.reforme1}). Thus, the dual problem contains a finite number of variables, hinting that the problem is perhaps tractable in the dual space \cite{gatsis2010,ribeiro2010}. Furthermore, we will demonstrate in Section V that after a careful algorithmic design in Section IV, our online solution for the relaxed problem \eqref{eq.reform2} is also feasible for  \eqref{eq.prob}.

\subsection{Lagrange dual and optimal solutions}\label{subsec.Lag-dual}
Let $\bm{\lambda}:=[\lambda_{1}^{\rm mn},\ldots,\lambda_{J}^{\rm mn}, \lambda_{1}^{\rm dc}, \ldots, \lambda_{I}^{\rm dc}]^{\top}$ denote the $(I+J)\times 1$ Lagrange multiplier vector associated with constraints \eqref{eq.reforme1}. Upon defining $\mathbf{x}:= \{\mathbf{x}_t, \forall t\}$, the partial Lagrangian function of \eqref{eq.reform} is ${\cal L}(\mathbf{x},\bm{\lambda}) \!:=\mathbb{E}\big[{\cal L}_t(\mathbf{x}_t,\bm{\lambda})\big]$, where the instantaneous Lagrangian is given by\begin{align}\label{eq.Lam}
{\cal L}_t(\mathbf{x}_t,\bm{\lambda}) \!:=&\Psi_t(\mathbf{x}_t)+\bm{\lambda}^{\top}(\mathbf{A}\mathbf{x}_t+\mathbf{c}_t).
\end{align}
Considering ${\cal X}$ as the feasible set specified by the instantaneous constraints in  \eqref{eq.reforme2}, which are not dualized in \eqref{eq.Lam}, the dual function ${\cal D}(\bm{\lambda})$ can be written as
\begin{equation}\label{eq.dual-func}
{\cal D}(\bm{\lambda}):=\min_{\{\mathbf{x}_t \in {\cal X},\,\forall t\}}\!{\cal L}(\mathbf{x},\bm{\lambda}):=\mathbb{E}\Big[\min_{\mathbf{x}_t \in {\cal X}}{\cal L}_t(\mathbf{x}_t,\bm{\lambda})\Big].
\end{equation}
Correspondingly, the dual problem of \eqref{eq.reform} is
\begin{align}\label{eq.dual-prob}
 \max_{\bm{\lambda}\geq 0} \, {\cal D}(\bm{\lambda}):=\mathbb{E}\left[{\cal D}_t(\bm{\lambda})\right]
\end{align}
where ${\cal D}_t(\bm{\lambda}):=\min_{\mathbf{x}_t \in {\cal X}}{\cal L}_t(\mathbf{x}_t,\bm{\lambda})$. We henceforth refer to \eqref{eq.dual-prob} as the ensemble dual problem.
Note that similar to $\Psi_t$, here ${\cal D}_t$ and ${\cal L}_t$ are both parameterized by $\mathbf{s}_t$.

If the optimal Lagrange multipliers $\bm{\lambda}^*$ were known, a sufficient condition for the optimal solution of \eqref{eq.reform} or \eqref{eq.reform2} would be to minimize the Lagrangian ${\cal L}(\mathbf{x},\bm{\lambda}^*)$ or its instantaneous versions $\{{\cal L}_t(\mathbf{x}_t,\bm{\lambda}^*)\}$ over the set ${\cal X}$ \cite[Proposition 3.3.4]{bertsekas1999}.
Specifically, as formalized in the ensuing proposition, the optimal routing $\{\tilde{\mathbf{x}}_{j,t}^*,\forall j\}$ and workload $\{x_{i,t}^*,\forall i\}$ schedules in this MN-DC network can be expressed as a function of $\bm{\lambda}^*$ associated with \eqref{eq.reforme1}, and the realization of the random state $\mathbf{s}_t$.

\begin{proposition}\label{prop.closedform}
Consider the strictly convex costs $\nabla G_{i,j}^{\rm d}$ and $\nabla G_{i,t}^{\rm e}$ in \eqref{eq.net}.
Given the realization $\mathbf{s}_t$ in \eqref{eq.reform2}, and the Lagrange multipliers $\bm{\lambda}^*$ associated with \eqref{eq.reforme1}, the optimal instantaneous workload routing decisions are
\begin{subequations}\label{close-sol}
	\begin{equation}\label{close-up1}
		~~~~~~~\tilde{x}_{i,j,t}^*(\mathbf{s}_t)=\big[(\nabla G_{i,j}^{\rm d})^{-1}\big((\lambda_{j}^{\rm mn})^*-(\lambda_{i}^{\rm dc})^*\big)\big]_0^{B_{i,j}}
	\end{equation}
 and the optimal instantaneous workload scheduling decisions are given by [with $G_{i}^{\rm e}(\mathbf{x}_t;\mathbf{s}_t):=G_{i,t}^{\rm e}(\mathbf{x}_t)$]
	\begin{equation}\label{close-up2}
		x_{i,t}^*(\mathbf{s}_t)=\big[\big(\nabla G_{i}^{\rm e}(\mathbf{s}_t)\big)^{-1}\big((\lambda_{i}^{\rm dc})^*\big)\big]_0^{D_i}
	\end{equation}
\end{subequations}
where $(\nabla G_{i,j}^{\rm d})^{-1}$ and $\big(\nabla G_{i}^{\rm e}(\mathbf{s}_t)\big)^{-1}$ denote the inverse functions of $\nabla G_{i,j}^{\rm d}$ and $\nabla G_{i,t}^{\rm e}$, respectively.

\end{proposition}

We omit the proof of Proposition \ref{prop.closedform}, which can be easily derived using the KKT conditions for constrained optimization \cite{bertsekas1999}.
Building on Proposition \ref{prop.closedform}, it is interesting to observe that the stationary policy we are looking for in Section \ref{subsec.reform} is expressible uniquely in terms of $\bm{\lambda}^*$.
The intuition behind this solution is that the Lagrange multipliers act as interfaces between MN-DC and workload-power balance, capturing the availability of resources and utility information which is relevant from a resource allocation point of view.

However, to implement the optimal resource allocation in \eqref{close-sol}, the optimal multipliers $\bm{\lambda}^*$ must be known. To this end, we first outline the celebrated stochastic approximation-based and corresponding Lyapunov approaches to stochastic network optimization.
Subsequently, we develop a novel approach in Section \ref{sec.SAGA} to learn the optimal multipliers in both offline and online settings.

\subsection{Stochastic dual (sub-)gradient ascent}\label{subsec.DGD}
For the ensemble dual problem \eqref{eq.dual-prob}, a standard (sub)gradient iteration involves taking the expectation over the distribution of $\mathbf{s}_t$ to compute the gradient \cite{gatsis2010}.
This is challenging because the underlying distribution of $\mathbf{s}_t$ is usually unknown in practice.
Even if the joint probability distribution functions were available, finding the expectations can be non-trivial especially in high-dimensional settings ($I\gg 1$ and/or $J\gg 1$).

To circumvent this challenge, a popular solution relies on stochastic approximation  \cite{robbins1951,neely2010,chen2016}. The resultant stochastic dual gradient (SDG) iterations can be written as (cf. \eqref{eq.dual-prob})
\begin{align}\label{eq.dual-stocg}
\check{\bm{\lambda}}_{t+1} ~=\big[ \check{\bm{\lambda}}_t + \mu \nabla{\cal D}_t(\check{\bm{\lambda}}_t)\big]^{+}
\end{align}
where the stochastic (instantaneous) gradient $\nabla{\cal D}_t(\check{\bm{\lambda}}_t)=\mathbf{A}\mathbf{x}_t+\mathbf{c}_t$ is an unbiased estimator of the ensemble gradient given by $\mathbb{E}[\nabla{\cal D}_t(\check{\bm{\lambda}}_t)]$. The primal variables $\mathbf{x}_t$ can be found by solving ``on-the-fly'' the instantaneous problems, one per slot $t$
\begin{align}\label{eq.SA-sub}
	\mathbf{x}_t\in\arg\min_{\mathbf{x}_t \in {\cal X}}{\cal L}_t(\mathbf{x}_t,\check{\bm{\lambda}}_t)
\end{align}
where the operator $\in$ accounts for cases that the Lagrangian has multiplier minimizers.
The minimization in \eqref{eq.SA-sub} is not difficult to solve. For a number of relevant costs and utility functions in Proposition 1, closed-form solutions are available for the primal variables.
Note from \eqref{eq.dual-stocg} that the iterate $\check{\bm{\lambda}}_{t+1}$ depends only on the probability distribution of $\mathbf{s}_t$ through the stochastic gradient $\nabla{\cal D}_t(\check{\bm{\lambda}}_t)$. Consequently, the process $\{\check{\bm{\lambda}}_t\}$ is Markov with transition probability that is time invariant since $\mathbf{s}_t$ is stationary. In the context of Lyapunov optimization \cite{Urg11,huang2011} and \cite{neely2010}, the Markovian iteration $\{\check{\bm{\lambda}}_t\}$ in \eqref{eq.dual-stocg} is interpreted as a virtual queue recursion; i.e., $\check{\bm{\lambda}}_t=\mu\mathbf{q}_t$.

Thanks to their low complexity and ability to cope with non-stationary scenarios, SDG-based approaches are widely used in various research disciplines; e.g., adaptive signal processing \cite{kong1995}, stochastic network optimization \cite{neely2010,huang2011,huang2014,chen2017tpds}, and energy management in power grids \cite{Urg11,sun2016}.
Unfortunately, SDG iterates are known to converge slowly.
Although simple, SDG does not exploit the often sizable number of historical samples.
These considerations motivate a systematic design of an offline-aided-online approach, which can significantly improve online performance of SDG for constrained optimization, and can have major impact in e.g.,  network resource allocation tasks by utilizing streaming big data, while preserving its low complexity and fast adaptation.

\section{Learn-and-adapt Resource allocation}\label{sec.SAGA}

Before developing such a promising approach that we view as learn-and-adapt SDG scheme, we list our
assumptions that are satisfied in typical network resource allocation problems.

\noindent\textbf{(as1)} \textit{State process $\{\mathbf{s}_t\}$ is independent and identically distributed (i.i.d.), and the common probability density function (pdf) has bounded support.}

\noindent\textbf{(as2)} \textit{The cost $\Psi_t(\mathbf{x}_t)$ in \eqref{eq.net} is non-decreasing w.r.t. $\mathbf{x}_t\in{\cal X}:=\{\mathbf{0}\leq \mathbf{x}_t \leq \mathbf{\bar{x}}\}$; it is a $\sigma$-strongly convex function\footnote{We say that a function $f: {\rm dom}(f) \rightarrow \mathbb{R}$ is $\sigma$-strongly convex if and only if $f(\mathbf{x})-\frac{\sigma}{2} \|\mathbf{x}\|^2$ is convex for all $\mathbf{x}\in {\rm dom}(f)$, where ${\rm dom}(f)\subseteq\mathbb{R}^n$ \cite{bertsekas1999}.}; and its gradient is Lipschitz continuous with constant $\tilde{L}$ for all $t$. }

\noindent\textbf{(as3)} \textit{There exists a stationary policy $\bm{\chi}(\cdot)$ satisfying $0\leq\bm{\chi}(\mathbf{s}_t) \leq \mathbf{\bar{x}}$, for all $\mathbf{s}_t$, and $\mathbb{E}[\mathbf{A}\bm{\chi}(\mathbf{s}_t)+\mathbf{c}_t(\mathbf{s}_t)]\leq -\bm{\zeta}$, where $\bm{\zeta}>\mathbf{0}$ is the slack constant}; and

\noindent\textbf{(as4)} \textit{The instantaneous dual function ${\cal D}_t(\bm{\lambda})$ in \eqref{eq.dual-prob} is $\epsilon$-strongly concave, and its gradient $\nabla{\cal D}_t(\bm{\lambda})$ is $L$-Lipschitz continuous with condition number $\kappa=L/\epsilon$, for all $t$.}

Although {(as1)} can be relaxed if ergodicity holds, it is typically adopted by stochastic resource allocation schemes for simplicity in exposition \cite{huang2011,huang2014,eryilmaz2006}.
Under {(as2)}, the objective function is non-decreasing and strongly convex, which is satisfied in practice with quadratic/exponential utility or cost functions \cite{eryilmaz2006}.
The so-called Slater's condition in {(as3)} ensures the existence of a bounded Lagrange multiplier \cite{bertsekas1999}, which is necessary for the queue stability of \eqref{eq.prob}; see e.g., \cite{Geor06,huang2014}.
If (as3) cannot be satisfied, one should consider reducing the workload arrival rates at the MN side, or, increasing the link and facility capacities at the DC side.
The $L$-Lipschitz continuity of $\nabla{\cal D}_t(\bm{\lambda})$ in (as4) directly follows from the strong-convexity of the primal function in {(as2)} with $L=\rho(\mathbf{A}^{\top}\mathbf{A})/\sigma$, where $\rho(\mathbf{A}^{\top}\mathbf{A})$ is the spectral radius of the matrix $\mathbf{A}^{\top}\mathbf{A}$.
The strong concavity in (as4) is frequently assumed in network optimization \cite{liu2016}, and it is closely related to the local smoothness and the uniqueness of Lagrange multipliers assumed in \cite{huang2011,huang2014,eryilmaz2006}.
For pessimistic cases, {(as4)} can be satisfied by subtracting an $\ell_2$-regularizer from the dual function \eqref{eq.dual-func}, which is typically used in machine learning applications (e.g., ridge regression).
We quantify its sub-optimality in Appendix A.
Note that Appendix A implies that the primal solution $\mathbf{x}_t^*$ will be ${\cal O}(\sqrt{\epsilon})$-optimal and feasible for the regularizer $({\epsilon}/{2})\|\bm{\lambda}\|^2$. Since we are after an ${\cal O}(\mu)$-optimal online solution, it suffices to set $\epsilon={\cal O}(\mu)$.

\subsection{Batch learning via offline SAGA based training}\label{subsec.offSAGA}

Consider that a training set of $N$ historical state samples ${\cal S}:=\{\mathbf{s}_n,1\leq n\leq N\}$ is available. Using ${\cal S}$, we can find an empirical version of \eqref{eq.dual-func} via sample averaging as
\begin{equation}\label{eq.dual-func2}
\hat{{\cal D}}_{{\cal S}}(\bm{\lambda}):=\frac{1}{N}\sum_{n=1}^N\hat{{\cal D}}_n(\bm{\lambda})=\frac{1}{N}\sum_{n=1}^N\Big[\min_{\mathbf{x}_n \in {\cal X}}{\cal L}_n(\mathbf{x}_n,\bm{\lambda})\Big].
\end{equation}
Note that $t$ has been replaced by $n$ to differentiate training (based on historical data) from operational (a.k.a. testing or tracking) phases.
Consequently, the empirical dual problem can be expressed as
\begin{align}\label{eq.dual-prob2}
 \max_{\bm{\lambda}\geq \mathbf{0}} \, \frac{1}{N}\sum_{n=1}^N\hat{{\cal D}}_n(\bm{\lambda}).
\end{align}

Recognizing that the objective is a sum of finite concave functions, the task in \eqref{eq.dual-prob2} in the machine learning parlance is termed empirical risk maximization (ERM) \cite{vapnik2013}, which is carried out using
the batch gradient ascent iteration
\begin{equation}\label{dual-gd}
	 \bm{\lambda}_{k+1}=\left[\bm{\lambda}_k\!+\!\frac{\eta}{N} \sum_{n=1}^N\nabla\hat{{\cal D}}_n(\bm{\lambda}_k)\right]^{+}\!
\end{equation}
where the index $k$ represents the batch learning (iteration) index, and $\eta$ is the stepsize that controls the learning rate.
While iteration \eqref{dual-gd} exhibits a decent convergence rate, its computational complexity will be prohibitively high as the data size $N$ grows large. A typical alternative is to employ a stochastic gradient (SG) iteration, which uniformly at random selects one of the summands in \eqref{dual-gd}. However, such an SG iteration relies only on a single unbiased gradient correction, which leads to a sub-linear convergence rate. Hybrids of stochastic with batch gradient methods are popular subjects recently \cite{defazio2014,roux2012}.\footnote{Stochastic iterations for the empirical dual problem are different from that in Section \ref{subsec.DGD}, since stochasticity is introduced by the randomized algorithm itself, in oppose to the stochasticity of future states in the online setting.}

Leveraging our special problem structure, we will adapt the recently developed stochastic average gradient approach (SAGA) to fit our dual space setup, with the goal of efficiently computing empirical Lagrange multipliers. Compared with the original SAGA that is developed for \textit{unconstrained} optimization \cite{defazio2014}, here we start from the \textit{constrained} optimization problem \eqref{eq.prob}, and derive first a projected form of SAGA.


\begin{algorithm}[t]
\caption{Offline SAGA iteration for batch learning}\label{algo}
\begin{algorithmic}[1]
\State \textbf{Initialize:} $\bm{\lambda}_0$, $k[n]=0,\,\forall n$, and stepsize $\eta$.
\For {$k=0,1,2\dots$}
\State Pick a sample index $\nu(k)$ uniformly at random from the {\color{white}~~~~}set $\{1,\ldots,N\}$.
\State Evaluate $\nabla\hat{{\cal D}}_{\nu(k)}(\bm{\lambda}_k)$, and update $\bm{\lambda}_{k+1}$ as in \eqref{eq.saga}.
\State Update the iteration index $k[n]\!=\!k$ for $\nu(k)=n$.
\EndFor
\end{algorithmic}
\end{algorithm}

Per iteration $k$, offline SAGA first evaluates at the current iterate $\bm{\lambda}_k$, one gradient sample $\nabla\hat{{\cal D}}_{\nu(k)}(\bm{\lambda}_k)$ with sample index $\nu(k)\in\{1,\ldots,N\}$ selected \textit{uniformly at random}. Thus, the computational complexity of SAGA is that of SG, and markedly less than the batch gradient ascent \eqref{dual-gd}, which requires $N$ such evaluations.
Unlike SG however, SAGA stores a collection of the most recent gradients $\nabla\hat{{\cal D}}_n(\bm{\lambda}_{k[n]})$ for all samples $n$, with the auxiliary iteration index $k[n]$ denoting the most recent past iteration that sample $n$ was randomly drawn; i.e., $k[n]:=\sup\{k':\nu(k')=n,\, 0\leq k'< k\}$.
Specifically, SAGA's gradient $\mathbf{g}_k(\bm{\lambda}_k)$ combines linearly the gradient $\nabla\hat{{\cal D}}_{\nu(k)}(\bm{\lambda}_k)$ randomly selected at iteration $k$ with the stored ones $\{\nabla\hat{{\cal D}}_n(\bm{\lambda}_{k[n]})\}_{n=1}^N$ to update the multipliers.
The resultant gradient $\mathbf{g}_k(\bm{\lambda}_k)$ is the sum of the
difference between the fresh gradient $\nabla\hat{{\cal D}}_{\nu(k)}(\bm{\lambda}_k)$ and the stored one $\nabla\hat{{\cal D}}_{\nu(k)}(\bm{\lambda}_{k[\nu(k)]})$ at the same sample, as well as the average of all gradients in the memory, namely
\vspace{-0.2cm}
\begin{subequations}\label{eq.saga}
	\begin{align}\label{eq.saga_grad}
    \mathbf{g}_k(\bm{\lambda}_k)\!=\!\nabla\hat{{\cal D}}_{\nu(k)}(\bm{\lambda}_k)\!-\!\nabla\hat{{\cal D}}_{\nu(k)}(\bm{\lambda}_{k[\nu(k)]})\!+\!\frac{1}{N}\!\sum_{n=1}^N\!{\nabla\hat{{\cal D}}_n(\bm{\lambda}_{k[n]})}.
  \end{align}
Therefore, the update of the offline SAGA can be written as
\begin{align}\label{eq.saga_iter}
\bm{\lambda}_{k+1}=\left[\bm{\lambda}_k+\eta \mathbf{g}_k(\bm{\lambda}_k)\right]^{+}
  \end{align}
\end{subequations}
  where $\eta$ denotes the stepsize. The steps of offline SAGA are summarized in Algorithm \ref{algo}.

To expose the merits of SAGA, recognize first that since $\nu(k)$ is drawn uniformly at random from set $\{1,\ldots,N\}$, we have that $\mathbb{P}\{\nu(k)=n\}=1/N$, and thus the expectation of the corresponding gradient sample is given by
\vspace{-0.2cm}
\begin{equation}
	\!\mathbb{E}[\nabla\hat{{\cal D}}_{\nu(k)}(\bm{\lambda}_k)]\!:=\!\!\sum_{n=1}^N\! \mathbb{P}\{\nu(k)=n\} \nabla\hat{{\cal D}}_n(\bm{\lambda}_k)\!=\!\frac{1}{N}\!\sum_{n=1}^N\!\nabla\hat{{\cal D}}_n(\bm{\lambda}_k).
\end{equation}
Hence, $\nabla\hat{{\cal D}}_{\nu(k)}(\bm{\lambda}_k)$ is an unbiased estimator of the empirical gradient in \eqref{dual-gd}. Likewise, $\mathbb{E}[\nabla\hat{{\cal D}}_{\nu(k)}(\bm{\lambda}_{k[\nu(k)]})]=(1/N)\sum_{n=1}^N\nabla\hat{{\cal D}}_n(\bm{\lambda}_{k[n]})$, which implies that the last two terms in \eqref{eq.saga_grad} disappear when taking the mean w.r.t. $\nu(k)$; and thus, SAGA's stochastic averaging gradient estimator $\mathbf{g}_k(\bm{\lambda}_k)$ is unbiased, as is the case with SG that only employs $\nabla\hat{{\cal D}}_{\nu(k)}(\bm{\lambda}_k)$.

With regards to variance, SG's gradient estimator has $\tilde{\sigma}_k^2:={\rm var}(\nabla\hat{{\cal D}}_{\nu(k)}(\bm{\lambda}_k))$, which can be scaled down using decreasing stepsizes (e.g., $\eta_k=1/\sqrt{k}$), to effect convergence of $\{\bm{\lambda}_k\}$ iterates in the mean-square sense \cite{robbins1951}. As can be seen from the correction term in \eqref{eq.saga_grad}, SAGA's gradient estimator has lower variance than $\tilde{\sigma}_k^2$. Indeed, the sum term of $\mathbf{g}_k(\bm{\lambda}_k)$ in \eqref{eq.saga_grad} is deterministic, and thus it has no effect on the variance. However, representing gradients of the same drawn sample $\nu(k)$, the first two terms are highly correlated, and their difference has variance considerably smaller than $\tilde{\sigma}_k^2$.
More importantly, the variance of stochastic gradient approximation vanishes as $\bm{\lambda}_k$ approaches the optimal argument $\bm{\lambda}_{\cal S}^*$ for \eqref{eq.dual-prob2}; see e.g., \cite{defazio2014}. This is in oppose to SG where the variance of stochastic approximation remains even if the iterates are close to the optimal solution. This \textit{variance reduction} property allows SAGA to achieve a linear convergence with constant stepsizes, which is not achievable for the SG method.

\begin{figure}[t]
\centering
\vspace{-0.4cm}
\includegraphics[height=0.3\textwidth]{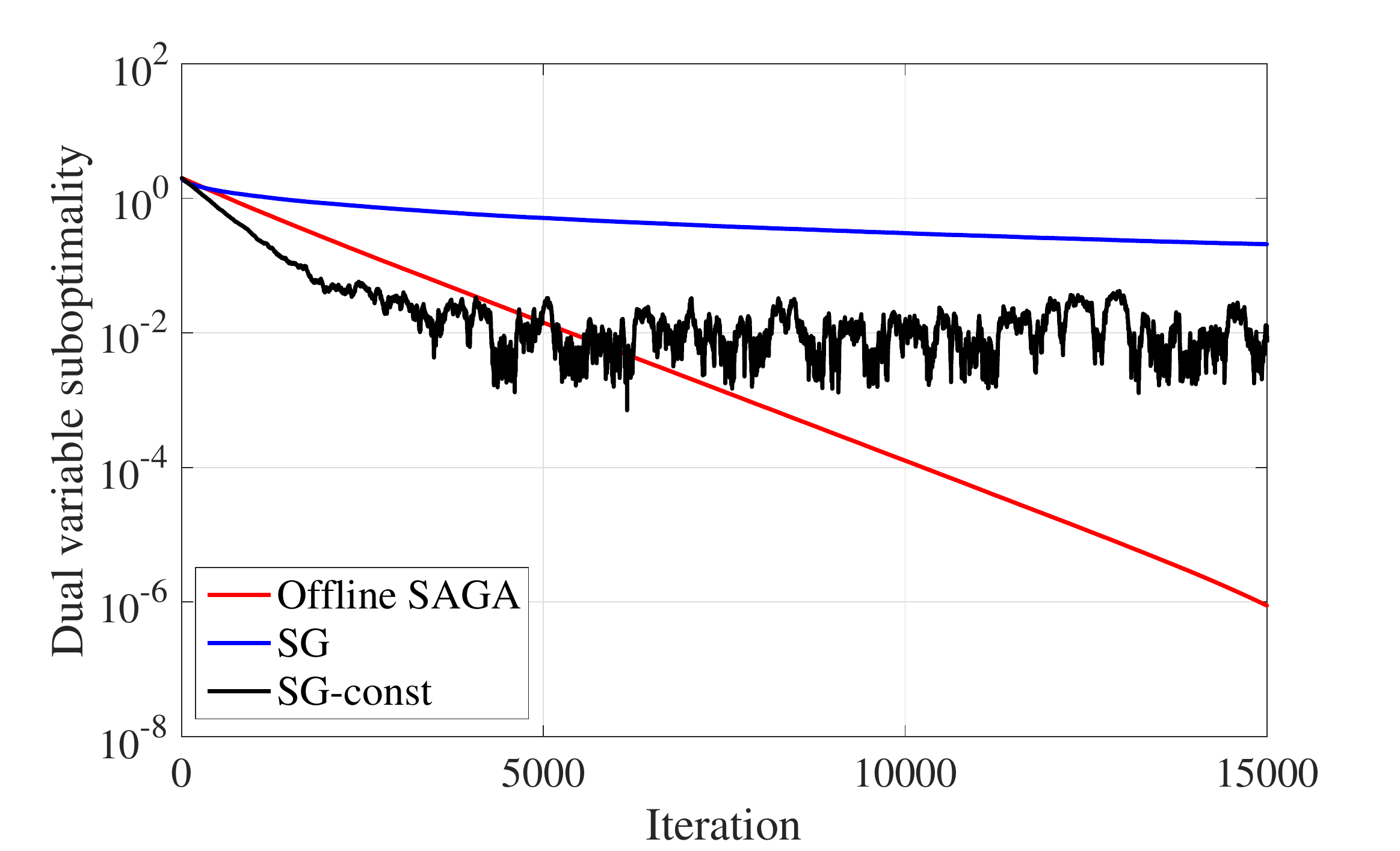}
\vspace{-0.6cm}
\caption{The dual sub-optimality $\|\bm{\lambda}_k-\bm{\lambda}_{\cal S}^*\|/\|\bm{\lambda}_{\cal S}^*\|$ vs. number of iterations for SAGA with $\eta=1/(3L)$, and SG with constant stepsize $\eta=0.2$ and diminishing stepsize $\eta_k=1/\sqrt{k}$. We consider 100 state samples $\{\mathbf{s}_n\}_{n=1}^{100}$ for a cloud network with 4 DCs and 4 MNs. The details about the distribution for generating $\mathbf{s}_n:=\{\alpha_{i,n},e_{i,n},P_{i,n}^{\rm r},v_{j,n},\forall i,j\}$ are provided in Section \ref{sec.Num}.
SAGA is the only method that converges linearly to the optimal argument.}
\label{Fig.convg}
\vspace{-0.4cm}
\end{figure}

In the following theorem, we use the result in \cite{defazio2014} to show that the offline SAGA method is linearly convergent.


\begin{theorem}\label{the.linear-rate}
Consider the offline SAGA iteration in \eqref{eq.saga}, and assume that the conditions in (as2)-(as4) are satisfied. If $\bm{\lambda}^*_{\cal S}$ denotes the unique optimal argument in \eqref{eq.dual-prob2}, and the stepsize is chosen as $\eta=1/(3L)$ with Lipschitz constant $L$ as in {(as4)}, then SAGA iterates initialized with $\bm{\lambda}_0$ satisfy
\begin{subequations}
	\begin{align}\label{eq.linear-rate}
			\mathbb{E}_{\nu}\|\bm{\lambda}_k-\bm{\lambda}^*_{\cal S}\|^2\leq(\Gamma_N)^kC_{\cal S}(\bm{\lambda}_0)
\end{align}
where $\Gamma_N\!:=\!1-\min(\frac{1}{4N},\frac{1}{3\kappa})$ with $\kappa$ denoting the  condition number in (as4), the expectation $\mathbb{E}_{\nu}$ is over all choices of the sample index $\nu(k)$ up to iteration $k$, and constant $C_{\cal S}(\bm{\lambda}_0)$ is
  \begin{align}\label{eq.linear-rate_2}
C_{\cal S}(\bm{\lambda}_0)&:=\|\bm{\lambda}_0-\bm{\lambda}^*_{\cal S}\|^2
\\&
\
-\frac{2N}{3L}\!\left[\hat{{\cal D}}_{\!\cal S}(\bm{\lambda}_0)-\hat{{\cal D}}_{\cal S}(\bm{\lambda}^*_{\cal S})-\langle \nabla\hat{{\cal D}}_{\cal S}(\bm{\lambda}^*_{\cal S}),\bm{\lambda}_0-\bm{\lambda}^*_{\cal S}\rangle\right].\nonumber
	\end{align}
\end{subequations}
\end{theorem}
\begin{proof}
See Appendix \ref{app.linear-rate}.
\end{proof}
Since $\Gamma_N<1$, Theorem \ref{the.linear-rate} asserts that the sequence $\{\bm{\lambda}_k\}$ generated by SAGA converges exponentially to the empirical optimum $\bm{\lambda}^*_{\cal S}$ in mean-square.
Similar to ${\cal D}_t(\bm{\lambda})$ in {(as4)}, if $\{\hat{{\cal D}}_n(\bm{\lambda})\}_{n=1}^N$ in \eqref{eq.dual-func2} are $L$-smooth, then so is $\hat{{\cal D}}_{{\cal S}}(\bm{\lambda})$, and the sub-optimality gap induced by $\hat{{\cal D}}_{{\cal S}}(\bm{\lambda}_k)$ can be bounded by
\begin{equation}\label{eq.linear-fun}
	\mathbb{E}_{\nu}[\hat{{\cal D}}_{{\cal S}}(\bm{\lambda}^*_{\cal S})\!-\!\hat{{\cal D}}_{{\cal S}}(\bm{\lambda}_k)]\!\leq\! L\mathbb{E}_{\nu}\|\bm{\lambda}_k-\bm{\lambda}^*_{\cal S}\|^2\!\leq\! (\Gamma_N)^k LC_{\cal S}(\bm{\lambda}_0).\!
\end{equation}
Remarkably, our offline SAGA based on training data is able to obtain the order-optimal convergence rate among all first-order approaches at the cost of only a single gradient evaluation per iteration. We illustrate the convergence of offline SAGA with two SG variants in Fig. \ref{Fig.convg} for a simple example.
As we observe, SAGA converges to the optimal argument linearly, while the SG method with diminishing stepsize has a sublinear convergence rate and the one with constant stepsize converges to a neighborhood of the optimal solution.


\begin{remark}
The stepsize in Theorem \ref{the.linear-rate} requires knowing the Lipschitz constant of \eqref{eq.dual-func2}.
In our context, it is given by $L=\rho(\mathbf{A}^{\top}\mathbf{A})/\sigma$, with $\rho(\mathbf{A}^{\top}\mathbf{A})$ denoting the spectral radius of the matrix $\mathbf{A}^{\top}\mathbf{A}$, and $\sigma$ defined in (as2).
When $L$ can be properly approximated in practice, it is worth mentioning that the linear convergence rate of offline SAGA in Theorem \ref{the.linear-rate} can be established under a wider range of stepsizes, with slightly different constants $\Gamma_N$ and $C_{\cal S}(\bm{\lambda}_0)$ \cite[Theorem 1]{defazio2014}.
\end{remark}

\vspace{-0.2cm}

\subsection{Learn-and-adapt via online SAGA}\label{subsec.onSAGA}
Offline SAGA is clearly suited for learning from (even massive) training datasets available in batch format, and it is tempting to carry over our training-based iterations to the operational phase as well.
However, different from common machine learning tasks, the training-to-testing procedure is no longer applicable here, since the online algorithm must also track system dynamics to ensure the stability of queue lengths in our network setup.
Motivated by this observation,
we introduce a novel data-driven approach called \textit{online} SAGA, which incorporates the benefits of batch training to mitigate online stochasticity, while preserving the adaptation capability.

Unlike the offline SAGA iteration that relies on a fixed training set, here we deal with a dynamic learning task where the training set grows in each time slot $t$. Running sufficiently many SAGA iterations for each empirical dual problem per slot could be computationally expensive. Inspired by the idea of dynamic SAGA with a fixed training set \cite{Daneshmand16}, the crux of our online SAGA in the operational phase is to learn from the dynamically growing training set at an affordable computational cost. This allows us to obtain a reasonably accurate solution at an affordable computational cost - only a few SAGA iterations.

To start, it is useful to recall a learning performance metric that uniformly bounds the difference between the empirical loss
in \eqref{eq.dual-func2} and the ensemble loss in \eqref{eq.dual-func} with high probability (w.h.p.), namely \cite{vapnik2013}
\begin{equation}\label{ineq.gene}
	\sup_{\bm{\lambda}\geq \mathbf{0}}|{\cal D}(\bm{\lambda})-\hat{{\cal D}}_{{\cal S}}(\bm{\lambda})|\leq {\cal H}_s(N)
\end{equation}
where ${\cal H}_s(N)$ denotes a bound on the statistical error induced by the finite size $N$ of the training set ${\cal S}$.
Under proper (so-termed mixing) conditions, the law of large numbers guarantees that ${\cal H}_s(N)$ is in the order of ${\cal O}(\sqrt{1/N})$, or, ${\cal O}(1/N)$ for specific function classes \cite[Section 3.4]{vapnik2013}.
That the statement in \eqref{ineq.gene} holds with w.h.p. means that there exists a constant $\delta>0$ such that \eqref{ineq.gene} holds with probability at least $1-\delta$. In such case, the statistical accuracy ${\cal H}_s(N)$ depends on $\ln(1/\delta)$, but we keep that dependency implicit to simplify notation \cite{vapnik2013,bousquet2008}.

Let $\bm{\lambda}_{KN}$ be an ${\cal H}_o(KN)$-optimal solution obtained by the offline SAGA over the $N$ training samples, where ${\cal H}_o(KN)$ will be termed optimization error emerging due to e.g., finite iterations say on average $K$ per sample; that is, $\hat{{\cal D}}_{{\cal S}}(\bm{\lambda}_{\cal S}^*)-\hat{{\cal D}}_{{\cal S}}(\bm{\lambda}_{KN})\leq {\cal H}_o(KN)$, where $\bm{\lambda}_{\cal S}^*$ is the optimal argument in \eqref{eq.dual-prob2}.
Clearly, the \textit{overall learning error}, defined as the difference between the empirical loss with $\bm{\lambda}=\bm{\lambda}_{KN}$ and the ensemble loss with $\bm{\lambda}=\bm{\lambda}^*$, is bounded above w.h.p. as
 \begin{align}\label{eq.overall-error}
 	{\cal D}(\bm{\lambda}^*)\!-\!\hat{{\cal D}}_{{\cal S}}(\bm{\lambda}_{KN})=&{\cal D}(\bm{\lambda}^*)\!-\!\hat{{\cal D}}_{{\cal S}}(\bm{\lambda}^*)\!+\!\hat{{\cal D}}_{{\cal S}}(\bm{\lambda}^*)\!-\!\hat{{\cal D}}_{{\cal S}}(\bm{\lambda}_{KN})\nonumber\\
 	\leq &{\cal D}(\bm{\lambda}^*)\!-\!\hat{{\cal D}}_{{\cal S}}(\bm{\lambda}^*)\!+\!\hat{{\cal D}}_{{\cal S}}(\bm{\lambda}_{{\cal S}}^*)\!-\!\hat{{\cal D}}_{{\cal S}}(\bm{\lambda}_{KN})\nonumber\\
     \leq &{\cal H}_s(N)\!+\!{\cal H}_o(KN)
 \end{align}
where the bound superimposes the optimization error with the statistical error.
If $N$ is relatively small, ${\cal H}_s(N)$ will be large, and keeping the per sample iterations $K$ small is reasonable for reducing complexity, but also because ${\cal H}_o(KN)$ stays comparable to ${\cal H}_s(N)$. With the dynamically growing training set ${\cal S}_t\!:=\!{\cal S}_{t-1}\cup \mathbf{s}_t$ in the operational phase, our online SAGA will aim at a ``sweet-spot'' between affordable complexity (controlled by $K$), and desirable overall learning error that is bounded by ${\cal H}_s(N_t)+{\cal H}_o(KN_t)$ with $N_t=|{\cal S}_t|$.


\begin{algorithm}[t]
\caption{Online SAGA for learning-while-adapting}\label{algo2}
\begin{algorithmic}[1]
\State \textbf{Offline initialize:} batch training set ${\cal S}_{\rm off}:=\{\mathbf{s}_n\}_{n=1}^{N_{\rm off}}$; initial iterate $\bm{\lambda}_0(0)$; iteration index $k[n]=0,\,\forall n$, and stepsize $\eta=1/(3L)$.

\State \textbf{Offline learning:} run offline SAGA in Algorithm \ref{algo} for $KN_{\rm off}$ iterations using historical samples in ${\cal S}_{\rm off}$.
\State \textbf{Online initialize:} hot start with $\bm{\lambda}_0(1)=\bm{\lambda}_{KN_{\rm off}}(0)$, $\{\nabla\hat{{\cal D}}_n(\bm{\lambda}_{k[n]}(0))\}_{n=1}^{N_{\rm off}}$ from the offline SAGA output using ${\cal S}_0={\cal S}_{\rm off}$; queue length $\mathbf{q}_0$; control variables $\mu>0$, and $\mathbf{b}:=[\sqrt{\mu}\log^2(\mu),\dots,\sqrt{\mu}\log^2(\mu)]^{\top}$.
\For {$t=1,2\dots$}
\State \textbf{Online adaptation:}
\State Compute effective multiplier using \eqref{eq.effect-dual}; acquire $\mathbf{s}_t$; and use $\bm{\gamma}_t$ to obtain $\mathbf{x}_t$ as in \eqref{close-up1}-\eqref{close-up2}, or by solving \eqref{eq.real-time}.
\State Update $\mathbf{q}_{t+1}$ using \eqref{W-delay} and \eqref{Q-delay}.
\State \textbf{Online learning:}
\State Form ${\cal S}_t\!:=\!{\cal S}_{t-1}\cup \mathbf{s}_t$ with cardinality $N_t\!=\!N_{t-1}\!+\!1$, and initialize the gradient entry for $\mathbf{s}_t$.
\State Initialized by $\bm{\lambda}_t$ and stored gradients, run $K$ offline SAGA iterations in Algorithm \ref{algo} with sample randomly drawn from set ${\cal S}_t$ with $N_t$ samples, and output $\bm{\lambda}_{t+1}$.
\EndFor
\end{algorithmic}
\end{algorithm}

The proposed online SAGA method consists of two complementary stages: \textit{offline learning} and \textit{online learning and adaptation}.
In the offline training-based learning, it runs $KN_{\rm off}$ SAGA iteration \eqref{eq.saga} (cf. Algorithm \ref{algo}) on a batch dataset ${\cal S}_{\rm off}$ with $N_{\rm off}$ historical samples,  and on average $K$ iterations per sample.
In the online operational phase, initialized with the offline output $\bm{\lambda}_0(1)=\bm{\lambda}_{KN_{\rm off}}(0)$,\footnote{To differentiate offline and online iterations, let $\bm{\lambda}_{k}(0)$ denote $k$-th iteration in the offline learning, and $\bm{\lambda}_k(t),\,t\geq 1$ denote $k$-th iteration during slot $t$.} online SAGA continues the learning process by storing a ``hot'' gradient collection as in Algorithm \ref{algo}, and also adapts to the queue length $\mathbf{q}_t$ that plays the role of an instantaneous multiplier estimate analogous to $\check{\bm{\lambda}}_t=\mu\mathbf{q}_t$ in \eqref{eq.dual-stocg}.
Specifically, online SAGA keeps acquiring data $\mathbf{s}_t$ and growing the training set ${\cal S}_t$ based on which it learns $\bm{\lambda}_t$ online by running $K$ iterations \eqref{eq.saga} per slot $t$.
The last iterate of slot $t-1$ initializes the first one of slot $t$; that is, $\bm{\lambda}_0(t)=\bm{\lambda}_K(t-1)$.
The learned $\bm{\lambda}_0(t)$ captures a priori information ${\cal S}_{t-1}$.
To account for instantaneous state information as well as actual constraint violations which are valuable for adaptation, online SAGA will superimpose $\bm{\lambda}_0(t)$ to the queue length (multiplier estimate) $\mathbf{q}_t$.
To avoid possible bias $\mathbf{b}\in\mathbb{R}^I_+$ in the steady-state, this superposition is bias-corrected, leading to an effective multiplier estimate (with short-hand notation $\bm{\lambda}_t:=\bm{\lambda}_0(t)$)
\begin{equation}\label{eq.effect-dual}
	\underbrace{~~~~\bm{\gamma}_t~~~~}_{\rm effective~multiplier}=\underbrace{~~~~~\bm{\lambda}_t~~~~~}_{\rm statistical~learning}+\underbrace{~~~\mu\mathbf{q}_t-\mathbf{b}~~~}_{\rm stochastic~ approx.}
\end{equation}
where scalar $\mu$ tunes emphasis on past versus present state information, and the value of constant $\mathbf{b}$ will be specified in Theorem \ref{gap-onlineSAGA}.
Based on the effective dual variable, the primal variables are obtained as in \eqref{close-sol} with $\bm{\lambda}^*$ replaced by $\bm{\gamma}_t$; or, for general constrained optimization problems, as
\begin{equation}\label{eq.real-time}
	\mathbf{x}_t(\bm{\gamma}_t):=\arg\min_{\mathbf{x}_t \in {\cal X}}{\cal L}_t(\mathbf{x}_t,\bm{\gamma}_t).
\end{equation}
The necessity of combining statistical learning and stochastic approximation in \eqref{eq.effect-dual} can be further understood from obeying feasibility of the original problem \eqref{eq.prob}; e.g., constraints \eqref{eq.probm} and \eqref{eq.probn}.
Note that in the adaptation step of online SAGA, the variable $\mathbf{q}_t$ is updated according to \eqref{eq.probm}, and the primal variable $\mathbf{x}_t$ is obtained as the minimizer of the Lagrangian using the effective multiplier $\bm{\gamma}_t$, which accounts for the queue variable $\mathbf{q}_t$ via \eqref{eq.effect-dual}.
Assuming that $\mathbf{q}_t$ is approaching infinity, the variable $\bm{\gamma}_t$ becomes large, and, therefore, the primal allocation $\mathbf{x}_t(\bm{\gamma}_t)$ obtained by minimizing online Lagrangian \eqref{eq.real-time} with $\bm{\lambda}=\bm{\gamma}_t$ will make the penalty term $\mathbf{A}\mathbf{x}_t+\mathbf{c}_t$ negative. Therefore, it will lead to decreasing $\mathbf{q}_{t+1}$ through the recursion in \eqref{eq.probm}. This mechanism ensures that the constraint \eqref{eq.probn} will not be violated.\footnote{Without \eqref{eq.effect-dual}, for a finite time $t$, $\bm{\lambda}_t$ is always near-optimal for the ensemble problem \eqref{eq.dual-prob}, and the primal variable $\mathbf{x}_t(\bm{\lambda}_t)$ in turn is near-feasible w.r.t.  \eqref{eq.reforme1} that is necessary for \eqref{eq.probn}. Since $\mathbf{q}_t$ essentially accumulates online constraint violations of \eqref{eq.reforme1}, it will grow linearly with $t$ and eventually become unbounded.}
The online SAGA is listed in Algorithm \ref{algo2}.

\begin{remark}[\textbf{Connection with existing algorithms}]
Online SAGA has similarities and differences with common schemes in machine learning, and stochastic network optimization.

(P1) SDG in \cite{gatsis2010,chen2016} can be viewed as a special case of online SAGA, which purely relies on stochastic estimates of multipliers or instantaneous queue lengths, meaning $\bm{\gamma}_t=\mu\mathbf{q}_t$. In contrast, online SAGA further learns by stochastic averaging over (possibly big) data to improve convergence of SDG.

(P2) Several schemes have been developed for statistical learning via large-scale optimization, including  SG, SAG \cite{roux2012}, SAGA \cite{defazio2014}, dynaSAGA \cite{Daneshmand16}, and AdaNewton \cite{mokhtari2016}. However, these are generally tailored for \textit{static} large-scale optimization, and are not suitable for dynamic resource allocation, because having $\bm{\gamma}_t=\bm{\lambda}_t$ neither tracks constraint violations nor adapts to state changes, which is critical to ensure queue stability \eqref{eq.probn}.

(P3) While online SAGA shares the learning-while-adapting attribute with the methods in \cite{huang2014} and \cite{zhang2016}, these approaches generally require using histogram-based pdf estimation, and solve the resultant empirical dual problem per slot exactly. However, histogram-based estimation is not tractable for our high-dimensional continuous random vector; and when the dual problem cannot be written in an explicit form, solving it exactly per iteration is computationally expensive.
\end{remark}	

\begin{remark}[\textbf{Computational complexity}]

Considering the dimension of $\bm{\lambda}$ as $(I+J)\times 1$, the dual update in \eqref{eq.dual-stocg} has the computational complexity of ${\cal O}((I+J)\cdot\max\{I,J\})$, since each row of $\mathbf{A}$ has at most $\max\{I,J\}$ non-zero entries. As the dimension of $\mathbf{x}_t$ is $(IJ+I)\times 1$, the primal update in \eqref{eq.SA-sub}, in general, requires solving a convex program with the worst-case complexity of  ${\cal O}(I^{3.5}J^{3.5})$ using a standard interior-point solver.
Overall, the per-slot computational complexity of the Lyapunov optimization or SDG in \cite{neely2010,gatsis2010} is ${\cal O}(I^{3.5}J^{3.5})$.

Online SAGA has a comparable complexity with that of SDG during online adaptation phase, where the effective multiplier update, primal update and queue update also incur a worst-case complexity ${\cal O}(I^{3.5}J^{3.5})$ in total.
Although the online learning phase is composed of additional $K$ updates of the empirical multiplier $\bm{\lambda}_t$ through \eqref{eq.saga}, we only need to compute one gradient $\nabla\hat{{\cal D}}_{\nu}(\bm{\lambda}_t)$ at a randomly selected sample $\mathbf{s}_{\nu}$.
Hence, per-iteration complexity is determined by that of solving \eqref{eq.SA-sub}, and it is independent of the data size $N_t$.
Clearly, subtracting stored gradient $\nabla\hat{{\cal D}}_n(\bm{\lambda}_{k[n]})$ only incurs ${\cal O}(I+J)$ computations, and the last summand in \eqref{eq.saga_grad} can be iteratively computed using running averages.
 In addition, for each new $\mathbf{s}_t$ per slot, finding the initial gradient requires another ${\cal O}(I^{3.5}J^{3.5})$ computations.
Therefore, the overall complexity of online SAGA is still ${\cal O}(I^{3.5}J^{3.5})$, but with a (hidden) multiplicative constant $K+2$.
\end{remark}


\section{Optimality and Stability of Online SAGA}\label{sec:analysis}
Similar to existing online resource allocation schemes, analyzing the performance of online SAGA is challenging. In this section, we will first provide performance analysis for the online learning phase, upon which we assess performance of our novel resource allocation approach in terms of average cost and queue stability.
For the online learning phase of online SAGA, we establish an upper-bound for the optimization error as follows.

\begin{lemma}\label{error-rec}
Consider the proposed online SAGA in Algorithm \ref{algo2}, and suppose the conditions in (as1)-(as4) are satisfied. Define the statistical error bound ${\cal H}_s(N)=dN^{-\beta}$ where $d>0$ and $\beta\in (0,1]$. If we select $K\geq 6$ and $N_{\rm off}=3\kappa/4$, the optimization error  at time slot $t$ is bounded in the mean by
\begin{equation}\label{ineq.error-rec}
\small
	\mathbb{E}\left[\hat{{\cal D}}_{{\cal S}_t}(\bm{\lambda}_{{\cal S}_t}^*)\!-\!\hat{{\cal D}}_{{\cal S}_t}(\bm{\lambda}_t)\right]\leq c(K){\cal H}_s(N_t)+\frac{\xi}{e^{K/5}}\Big(\frac{3\kappa}{4N_t}\Big)^{\!\!\frac{K}{5}}
\end{equation}
where the expectation is over all source of randomness during the optimization; constants are defined as $N_t:=N_{\rm off}+t$ and $c(K):= 4/(K-5)$; and $\xi$ is the initial error of $\bm{\lambda}_0$  in \eqref{eq.linear-rate}, i.e., $\xi:=\mathbb{E}[C_{{\cal S}_{\rm off}}(\bm{\lambda}_0)]$ with expectation over the choice of ${\cal S}_{\rm off}$.
\end{lemma}
\begin{proof}
	See Appendix \ref{app.B}.	
\end{proof}

As intuitively expected, Lemma \ref{error-rec} confirms that for $N_{\rm off}$ fixed, increasing $K$ per slot will lower the optimization error in the mean; while for a fixed $K$, increasing $N_{\rm off}$ implies improved startup of the online optimization accuracy. However, solely increasing $K$ will incur higher computational complexity.
Fortunately, the next proposition argues that online SAGA can afford low-complexity operations.
\begin{proposition}\label{error-K=1}
Consider the assumptions and the definitions of $\xi$, ${\cal H}_s(N)$ and $\beta$ as those in Lemma \ref{error-rec}.
If we choose the average number of SAGA iterations per-sample $K=1$, and the batch size $N_{\rm off}=3\kappa/4$, the mean optimization error satisfies that
	\begin{equation}
	\mathbb{E}\!\left[\hat{{\cal D}}_{{\cal S}_t}(\bm{\lambda}_{{\cal S}_t}^*)\!-\!\hat{{\cal D}}_{{\cal S}_t}(\bm{\lambda}_t)\right]\!\!\leq\!c(1){\cal H}_s(N_t)\!+\xi \Big(\frac{6\kappa}{eN_t}\Big)^{\!\frac{8}{5}}\!
	\end{equation}
	with the constant defined as $c(1):=8^{\beta}\times11/6+4^{\beta}/2+2^{\beta-1}$.
\end{proposition}
\begin{proof}
	See Appendix \ref{app.prop2}.	
\end{proof}
Proposition \ref{error-K=1} shows that when $N_t$ is sufficiently large, the optimization accuracy with only $K=1$ iteration per slot (per new datum) is on the same order of the statistical accuracy provided by the current training set ${\cal S}_t$. 	
Note that for other choices of $K\in\{2,\ldots,5\}$, constant $c(K)$ can be readily obtained by following the proof of Proposition \ref{error-K=1}.
In principle, online SAGA requires $N_{\rm off}$ batch samples and $K$ iterations per slot to ensure that the optimization error is close to statistical accuracy. In practice, one can even skip the batch learning phase since online SAGA can operate with a ``cold start.''
Simulations in Section \ref{sec.Num} will validate that without offline learning, online SAGA can still markedly outperform existing approaches.

Before analyzing resource allocation performance, convergence of the empirical dual variables is in order.

\begin{theorem}\label{emp-dual}
Consider the proposed online SAGA in Algorithm \ref{algo2}, and suppose the conditions in (as1)-(as4) are satisfied.
If $N_{\rm off}\geq 3\kappa/4$, and $K\geq 1$ per sample, the empirical dual variables of online SAGA satisfy
	\begin{equation}\label{eq.emp-dual}
		\lim_{t\rightarrow \infty} \bm{\lambda}_t=\bm{\lambda}^*,\quad {\rm w.h.p.}
	\end{equation}
	where $\bm{\lambda}^*$ denotes the optimal dual variable for the ensemble minimization in  \eqref{eq.dual-prob}, and w.h.p. entails a probability arbitrarily close but strictly less than $1$.
\end{theorem}

\begin{proof}
Since the derivations for $K\in \{1,\dots,5\}$ follow the same steps with only different constant $c(K)$, we focus on $K\geq 6$.
Recalling the Markov's inequality that for any nonnegative random
variable $X$ and constant $\delta>0$, we have $\mathbb{P}(X\geq 1/\delta )\leq \mathbb{E}[X]\delta$.
We can rewrite \eqref{ineq.error-rec} as
\begin{equation}\label{ineq.error-rec-re}
	\mathbb{E}\left[\frac{\hat{{\cal D}}_{{\cal S}_t}(\bm{\lambda}_{{\cal S}_t}^*)\!-\!\hat{{\cal D}}_{{\cal S}_t}(\bm{\lambda}_t)}{c(K){\cal H}_s(N_t)+\left({3\kappa}/{(4N_t)}\right)^{\!K/5}{\xi}/{e^{\frac{K}{5}}}}\right]\leq 1.
\end{equation}
Since the random variable inside the expectation of \eqref{ineq.error-rec-re} is non-negative, applying Markov's inequality yields
\begin{equation}\label{ineq.markov}
    \mathbb{P}\!\left(\!\hat{{\cal D}}_{{\cal S}_t}(\bm{\lambda}_{{\cal S}_t}^*)\!-\!\hat{{\cal D}}_{{\cal S}_t}(\bm{\lambda}_t)\!\geq\! \frac{1}{\delta}\!\left[c(K){\cal H}_s(N_t)\!+\!\frac{\xi}{e^{\frac{K}{5}}}\Big(\frac{3\kappa}{4N_t}\Big)^{\!\!\frac{K}{5}}\right]\!\right)\!\leq\! \delta
\end{equation}
thus with probability at least $1-\delta$, it holds that
\begin{equation}\label{ineq.emp-error}
	\hat{{\cal D}}_{{\cal S}_t}(\bm{\lambda}_{{\cal S}_t}^*)-\hat{{\cal D}}_{{\cal S}_t}(\bm{\lambda}_t)\leq c(K){\cal H}_s(N_t)+\frac{\xi}{e^{K/5}}\Big(\frac{3\kappa}{4N_t}\Big)^{\!\!\frac{K}{5}}\!\!\!
\end{equation}
where the RHS of \eqref{ineq.emp-error} scales by the constant $1/{\delta}$, but again we keep that dependency implicit to simplify notation.

With the current empirical multiplier $\bm{\lambda}_t$, the optimality gap in terms of the ensemble objective \eqref{eq.dual-func} can be decomposed by
\begin{align}\label{total-error0}
	{\cal D}(\bm{\lambda}^*)\!-\!{\cal D}(\bm{\lambda}_t)\!=\!&\underbrace{{\cal D}(\bm{\lambda}^*)\!-\!\hat{{\cal D}}_{{\cal S}_t}(\bm{\lambda}_t)}_{{\cal R}_1}+\underbrace{\hat{{\cal D}}_{{\cal S}_t}(\bm{\lambda}_t)\!-\!{\cal D}(\bm{\lambda}_t)}_{{\cal R}_2}.
\end{align}
Note that ${\cal R}_1$ can be bounded above via \eqref{eq.overall-error} by
\begin{align}\label{ineq.emp-error2}
	{\cal R}_1&\stackrel{(a)}{\leq} \hat{{\cal D}}_{{\cal S}_t}(\bm{\lambda}_{{\cal S}_t}^*)\!-\!\hat{{\cal D}}_{{\cal S}_t}(\bm{\lambda}_t)+{\cal H}_s(N_t)\nonumber\\
	&\stackrel{(b)}{\leq} (1+c(K)){\cal H}_s(N_t)+\frac{\xi}{e^{K/5}}\Big(\frac{3\kappa}{4N_t}\Big)^{\!\!\frac{K}{5}}
\end{align}
where (a) uses \eqref{eq.overall-error} and the definition of ${\cal H}_o(KN_t)$; (b) follows from \eqref{ineq.emp-error}; and both (a) and (b) hold with probability $1-\delta$ so that \eqref{ineq.emp-error2} holds with probability at least $(1-\delta)^2\geq 1-2\delta$.

In addition, ${\cal R}_2$ can be bounded via the uniform convergence in \eqref{ineq.gene} with probability $1-\delta$ by
\begin{equation}\label{ineq.emp-error3}
	{\cal R}_2\leq \sup_{\bm{\lambda}\geq \mathbf{0}}|{\cal D}(\bm{\lambda})-\hat{{\cal D}}_{{\cal S}_t}(\bm{\lambda})|\leq {\cal H}_s(N)
\end{equation}
where the RHS scales by $\ln(1/\delta)$.
Plugging \eqref{ineq.emp-error2} and \eqref{ineq.emp-error3} into \eqref{total-error0}, with probability at least $1-3\delta$, it follows that
\begin{align}\label{total-error}
	\!\!	{\cal D}(\bm{\lambda}^*)\!-\!{\cal D}(\bm{\lambda}_t)\!\leq\! h(\delta)\Big[2\!+\!c(K)\Big]{\cal H}_s(N_t)\!+\!\frac{h(\delta)\xi}{e^{K/5}}\Big(\frac{3\kappa}{4N_t}\Big)^{\!\!\frac{K}{5}}\!\!\!
\end{align}
where the multiplicative constant $h(\delta)={\cal O}(\frac{1}{\delta}\ln{\frac{1}{\delta}})$ in the RHS is induced by the hidden constants in \eqref{ineq.emp-error2} and \eqref{ineq.emp-error3}, that is independent of $t$ but increasing as $\delta$ decreases.
Note that the strong concavity of the ensemble dual function ${\cal D}(\bm{\lambda})$ implies that $\frac{\epsilon}{2}\|\bm{\lambda}^*-\bm{\lambda}_t\|^2\leq {\cal D}(\bm{\lambda}^*)-{\cal D}(\bm{\lambda}_t)$; see \cite[Corollary 1]{yu2016}.
Together with \eqref{total-error}, with probability at least $1-3\delta$, it leads to
	\begin{align}\label{eq.duallimit}
		&\lim_{t\rightarrow \infty}\frac{\epsilon}{2}\|\bm{\lambda}^*-\bm{\lambda}_t\|^2\leq\lim_{t\rightarrow \infty}{\cal D}(\bm{\lambda}^*)-\!{\cal D}(\bm{\lambda}_t)\nonumber\\
		\leq &\lim_{t\rightarrow \infty} h(\delta)\left[2+c(K)\right]{\cal H}_s(N_t)\!+\!\frac{h(\delta)\xi}{e^{K/5}}\Big(\frac{3\kappa}{4N_t}\Big)^{\!\!\frac{K}{5}}\!\stackrel{(c)}{=}0
	\end{align}
	where (c) follows since $h(\delta)$ is a constant, $\lim_{t\rightarrow \infty} {\cal H}_s(N_t)=0$, and $\lim_{t\rightarrow \infty} {1}/{N_t}=0$.
	\end{proof}

Theorem \ref{emp-dual} asserts that $\bm{\lambda}_t$ learned through online SAGA converges to the optimal $\bm{\lambda}^*$ w.h.p., even for small $N_{\rm off}$ and $K$.
To increase the confidence of Theorem \ref{emp-dual}, one should decrease the constant $\delta$ to ensure a large $1-3\delta$.
	Although a small $\delta$ will also enlarge the multiplicative constant $h(\delta)$ in the RHS of \eqref{total-error}, the result in \eqref{eq.duallimit} still holds since ${\cal H}_s(N_t)$ and $1/N_t$ deterministically converge to null as $t\rightarrow \infty$. Therefore, Theorem \ref{emp-dual} indeed holds with arbitrary high probability with $\delta$ small enough. For subsequent discussions, we only use w.h.p. to denote this high probability $1-3\delta$ to simplify the exposition.
	
Considering the trade-off between learning accuracy and storage requirements, a remark is due at this point.
\begin{remark}[\textbf{Memory complexity}]
	Convergence to $\bm{\lambda}^*$ in \eqref{eq.emp-dual} requires $N_t\rightarrow\infty$, but in practice datasets are finite.
	Suppose that $N_{\rm off}\in[3\kappa/4,W]$, and online SAGA stores the most recent $W\leq N_t$ samples with their associated gradients from ${\cal S}_t$. Following Lemma \ref{error-rec} and \cite[Lemma 6]{mokhtari2016}, one can show that the statistical error term at the RHS of \eqref{total-error} will stay at ${\cal O}({\cal H}_s(W))$ for a sufficiently large $W$; thus, $\bm{\lambda}_t$ will converge to a ${\cal O}({\cal H}_s(W))$-neighborhood of $\bm{\lambda}^*$ w.h.p.
	However, since we eventually aim at an ${\cal O}({\sqrt{\mu}})$-optimal effective multiplier (cf. \eqref{eq.covg-gamma}), it suffices to choose $W$ such that ${\cal H}_s(W)=\mathbf{o}(\sqrt{\mu})$.
	\end{remark}

Having established convergence of $\bm{\lambda}_t$, the next step is to show that $\bm{\gamma}_t$ also converges to $\bm{\lambda}^*$, and being a function of $\bm{\gamma}_t$ the online resource allocation $\mathbf{x}_t$ is also asymptotically optimal (cf. Proposition \ref{prop.closedform}).
To study the trajectory of $\bm{\gamma}_t$, consider the iteration of successive differences (cf. \eqref{eq.effect-dual})
\begin{equation}\label{eq.gamma}
	\bm{\gamma}_{t+1}-\bm{\gamma}_t=(\bm{\lambda}_{t+1}-\bm{\lambda}_t)+\mu(\mathbf{q}_{t+1}-\mathbf{q}_t).
\end{equation}
In view of \eqref{eq.emp-dual} and \eqref{eq.gamma}, to establish convergence of $\bm{\gamma}_t$ it is prudent to study
the asymptotic behavior of $(\mathbf{q}_{t+1}-\mathbf{q}_t)$.
To this end, let $\tilde{\mathbf{b}}_t:=\bm{\lambda}^*-\bm{\lambda}_t+\mathbf{b}$ denote the error $\bm{\lambda}^*-\bm{\lambda}_t$ of the statistically learned dual variable plus a bias-control variable $\mathbf{b}$. Theorem \ref{emp-dual} guarantees that $\lim_{t\rightarrow \infty}\tilde{\mathbf{b}}_t=\mathbf{b},\,{\rm w.h.p}$ Next, we show that $\mathbf{q}_t$ is attracted towards $\tilde{\mathbf{b}}_t/\mu$.

\begin{lemma}\label{lem.drift}
Consider the online SAGA in Algorithm \ref{algo2}, and the conditions in (as1)-(as4) are satisfied.
There exists a constant $B=\Theta({1}/{\sqrt{\mu}})$ and a finite time $T_B<\infty$, such that for all $t\geq T_B$, if $\|\mathbf{q}_t-\tilde{\mathbf{b}}_t/\mu\|>B$, it holds w.h.p. that queue lengths under online SAGA satisfy
	\begin{equation}\label{eq.drift}
		\mathbb{E}\left[\left\|\mathbf{q}_{t+1}-\tilde{\mathbf{b}}_t/\mu\right\|\Big|\mathbf{q}_t\right]\leq \left\|\mathbf{q}_t-\tilde{\mathbf{b}}_t/\mu\right\|-\sqrt{\mu}.
	\end{equation}
\end{lemma}
\begin{proof}\label{pf.queue-drift}
See Appendix \ref{app.lemma2}.
\end{proof}
Lemma \ref{lem.drift} reveals that when $\mathbf{q}_t$ deviates from $\tilde{\mathbf{b}}_t/\mu$, it will bounce back towards $\tilde{\mathbf{b}}_t/\mu$ at the next slot. With the drift behavior of $\mathbf{q}_t$ in \eqref{eq.drift}, we are on track to establish the desired long-term queue stability.

\begin{theorem}\label{the.queue-stable}
Consider the online SAGA in Algorithm \ref{algo2}, and the conditions in (as1)-(as4) are satisfied.
With constants $\mathbf{b}$ and $\mu$ defined in \eqref{eq.effect-dual}, the steady-state queue length satisfies
\begin{equation}\label{eq.queue}
	\lim_{T\rightarrow \infty} \frac{1}{T} \sum_{t=1}^{T} \mathbb{E}\left[\mathbf{q}_t\right]=\frac{\mathbf{b}}{\mu}+{\cal O}(\frac{\mathbf{1}}{\sqrt{\mu}}),\;{\rm w.h.p.}
\end{equation}
where the probability is close but strictly less than $1$.
\end{theorem}
\begin{proof}\label{pf.queue-stable}
See Appendix \ref{app.theorem3}.
\end{proof}

Theorem \ref{the.queue-stable} entails that the online solution for the relaxed problem in \eqref{eq.reform2} is a feasible solution for the original problem in \eqref{eq.prob}, which justifies effectiveness of the online learn-and-adapt step \eqref{eq.effect-dual}.
Theorem \ref{the.queue-stable} also points out the importance of the parameter $\mathbf{b}$ in \eqref{eq.effect-dual}, since the steady-state queue lengths will hover around $\mathbf{b}/\mu$ with distance ${\cal O}(1/\sqrt{\mu})$.
Using Theorem \ref{emp-dual}, it suffices to show that $\bm{\gamma}_t$ converges to a neighborhood of $\bm{\lambda}^*$,
\begin{equation}\label{eq.covg-gamma}
	\lim_{t\rightarrow \infty}\bm{\gamma}_t=\bm{\lambda}^*+\mu\mathbf{q}_t-\mathbf{b}=\bm{\lambda}^*+{\cal O}(\sqrt{\mu})\cdot \mathbf{1},\;{\rm w.h.p.}
\end{equation}
Qualitatively speaking, online SAGA behaves similar to SDG in the steady state.
However, the
main difference is that through more accurate empirical estimate $\bm{\lambda}_t$, online SAGA is able to accelerate the convergence of $\bm{\gamma}_t$, and reduce network delay by replacing
$\mathbf{q}_t=\check{\bm{\lambda}}_t/\mu$ in SDG, with $\bm{\lambda}_t$ in \eqref{eq.effect-dual}.


On the other hand, since a smaller $\mathbf{b}$ corresponds to lower average delay, it is tempting to set $\mathbf{b}$ close to zero (cf. \eqref{eq.queue}). Though attractive, a small $\mathbf{b}$ ($\approx \mathbf{0}$) will contribute to an effective multiplier $\bm{\gamma}_t$ $(=\bm{\lambda}_t+\mu\mathbf{q}_t)$ always larger than $\bm{\lambda}^*$, since $\bm{\lambda}_t$ converges to $\bm{\lambda}^*$ and $\mathbf{q}_t \geq \mathbf{0}$. This introduces bias to the effective multiplier $\bm{\gamma}_t$, which means a larger reward in the context of resource allocation.
This in turn encourages ``over-allocation'' of resources, and thus leads to a larger optimality loss due to the non-decreasing property of the primal objective in (as2).
Using arguments similar to those in \cite{huang2011} and \cite{huang2014}, we formally establish next  that by properly selecting $\mathbf{b}$, online SAGA is asymptotically near-optimal.

\begin{theorem}\label{gap-onlineSAGA}
Under (as1)-(as4), let ${\Psi}^*$ denote the optimal objective value in \eqref{eq.prob} under any feasible schedule even  with future information available.
Choosing $\mathbf{b}=\sqrt{\mu}\log^2(\mu)\cdot \mathbf{1}$ with an appropriate $\mu$, the online SAGA in Algorithm \ref{algo2} yields a near-optimal solution for \eqref{eq.prob} in the sense that
\begin{equation}\label{eq.opt-gap}
	    \lim_{T\rightarrow \infty} \frac{1}{T} \sum_{t=1}^{T} \mathbb{E}\left[\Psi_t\left(\mathbf{x}_t(\bm{\gamma}_t)\right)\right] \leq {\Psi}^*+{\cal O}(\mu),\;{\rm w.h.p.}
\end{equation}
where $\mathbf{x}_t(\bm{\gamma}_t)$ denotes the real-time schedules obtained from the Lagrangian minimization \eqref{eq.real-time}, and again this probability is arbitrarily close but strictly less than $1$.
\end{theorem}

\begin{proof}
See Appendix \ref{app.theorem4}.
\end{proof}

Combining Theorems \ref{the.queue-stable} and \ref{gap-onlineSAGA}, and selecting $\mathbf{b}=\sqrt{\mu}\log^2(\mu)\cdot \mathbf{1}$, online SAGA is asymptotically near-optimal with average queue length ${\cal O}(\log^2(\mu)/{\sqrt{\mu}})$. This implies that in the context of stochastic network optimization \cite{neely2010}, the novel approach achieves a near-optimal cost-delay tradeoff $[\mu,\log^2(\mu)/{\sqrt{\mu}}]$ with high probability. Comparing with the standard tradeoff $[\mu,{1}/{\mu}]$ under SDG \cite{gatsis2010,neely2010}, the proposed offline-aided-online algorithm design proposed significantly improves the online performance in terms of delay in most cases.
This is a desired performance trade-off under the current setting in the sense that the ``optimal" tradeoff $[\mu,\log^2(\mu)]$ in \cite{huang2014} is derived under the local polyhedral assumption, which is typically the case when the primal feasible set contains a finite number of actions. The performance of online SAGA under settings with local polyhedral property is certainly of practical interest, but this goes beyond the scope of the present work, and we leave for future research.


\begin{remark}
The learning-while-adapting attribute of online SAGA amounts to choosing a scheduling policy, or equivalently effective multipliers, satisfying the following criteria.

C1) The effective dual variable is initiated or adjusted online close to the optimal multiplier enabling the online algorithm to quickly reach an optimal resource allocation strategy (cf. Prop. \ref{prop.closedform}). Unlike SDG that relies on incremental updates to adjust dual variables, online SAGA attains a near-optimal effective multiplier much faster than SDG, thanks to the stochastic averaging that accelerates convergence of statistical learning.

C2) Online SAGA judiciously accounts for queue dynamics in order to guarantee long-term queue stability, which becomes possible through instantaneous measurements of queue lengths.

Building on C1) and C2), the proposed data-driven constrained optimization approach can be extended to incorporate second-order iterations in the learning mode (e.g., AdaNewton in \cite{mokhtari2016}), or momentum iterations in the adaptation mode \cite{liu2016}.
\end{remark}

\section{Numerical Tests}\label{sec.Num}

This section presents numerical tests to confirm the analytical claims, and demonstrate the merits of the proposed approach.
The network considered has $I=4$ DCs, and $J=4$ MNs.
Performance is tested in terms of the time-average instantaneous network cost
\begin{equation}
\Psi_t(\mathbf{x}_t)\!:=\!\sum_{i\in{\cal I}}\alpha_{i,t}\left(e_{i,t}x_{i,t}^2 -P_{i,t}^{\rm r}\right)\!+\!\sum_{i\in{\cal I}}\sum_{j\in {\cal J}} c_{i,j}^{\rm d}\tilde{x}_{i,j,t}^2
\end{equation}
where the energy transaction price $\alpha_{i,t}$ is uniformly distributed over $[10,30]$ \$/kWh; the energy efficiency factors are time-invariant taking values $\{e_{i,t}\}=\{1.2,1.3,1.4,1.5\}$; samples of the renewable supply $\{P_{i,t}^{\rm r}\}$ are generated from a uniform distribution with support $[10,50]$ kWh; and the bandwidth cost is set to $c_{i,j}^{\rm d}=40/B_{i,j}$, with bandwidth limits $\{B_{i,j}\}$ generated from a uniform distribution with support $[10,100]$.
DC computing capacities are $\{D_i\}=\{200,150,100,100\}$,
and uniformly distributed over $[10,150]$ workloads $\{v_{j,t}\}$ arrive at each MN $j$. Unless specified otherwise, default parameter values are chosen as $\mathbf{b}=\sqrt{\mu}\log^2(\mu)\cdot \mathbf{1}$, and $\mu=0.1$.
Online SAGA is compared with two alternatives: a) the standard stochastic dual subgradient (SDG) algorithm in e.g., \cite{chen2016,neely2010}; and b) a `hot-started' version of SDG that we term SDG+, in which initialization is provided using offline SAGA with $N_{\rm off}$ training samples.
Note that the approaches in \cite{huang2014} and \cite{zhang2016} rely on histogram-based pdf estimation to approximate the ensemble dual problem \eqref{eq.dual-prob}.
For our setting however, directly estimating the pdf of a multivariate continuous random state $\mathbf{s}_t\in\mathbb{R}^{3I+J}$ is a considerably cumbersome work.
Even if we only discretize each entry of $\mathbf{s}_t$ into $10$ levels, the number of possible network states can be $10^{16}$ or $10^{80}$ in our simulated settings ($I\!=\!J\!=\!4$, or, $I\!=\!J\!=\!20$), thus they are not simulated for comparison.

\begin{figure}[t]
\centering
\vspace{-0.2cm}
\includegraphics[height=0.3\textwidth]{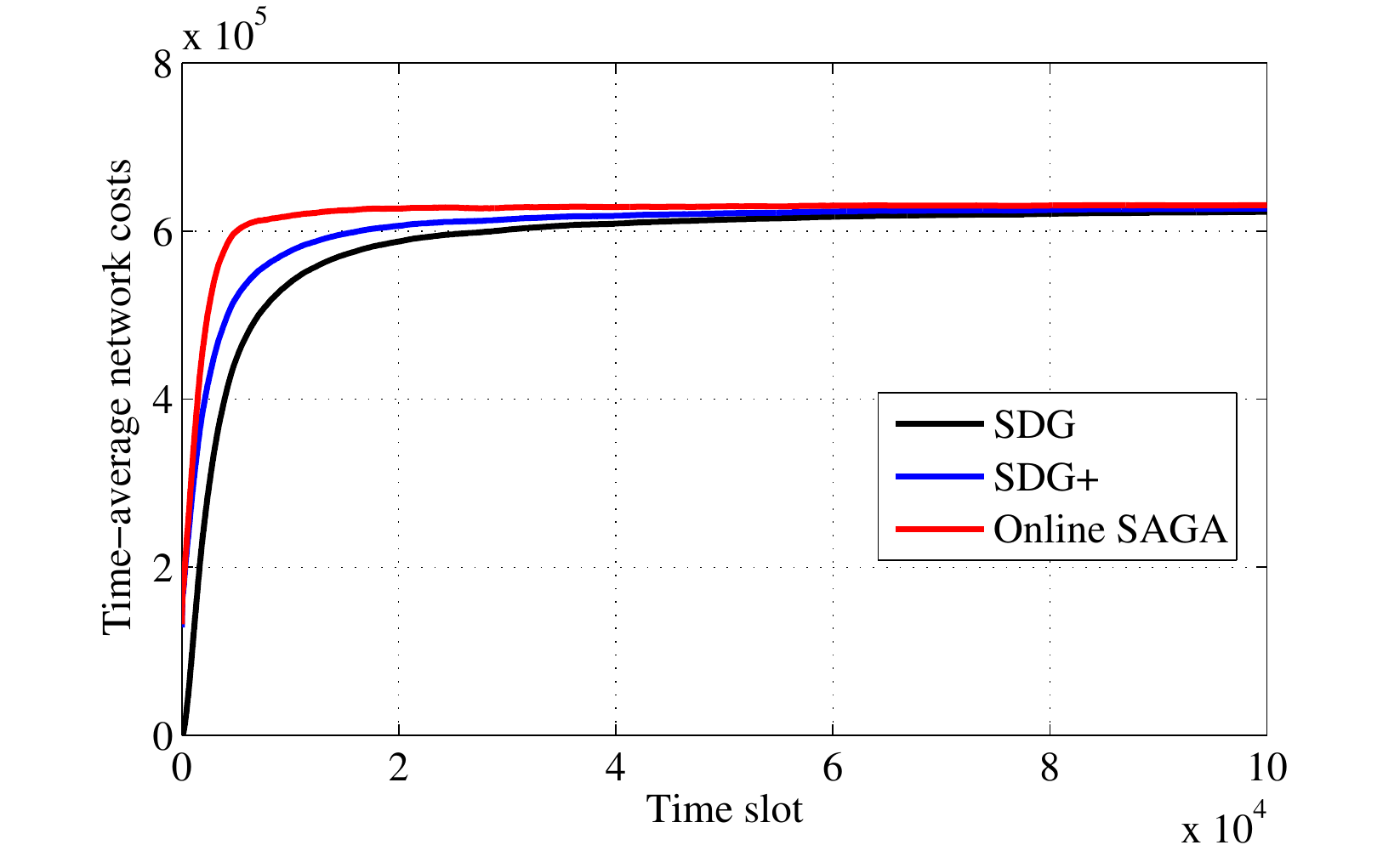}
\vspace{-0.6cm}
\caption{Comparison of time-average costs ($I=J=4$, $N_{\rm off}=1,000$, $K=2$)}
\label{Fig.obj}
\end{figure}

\begin{figure}[t]
\centering
\vspace{-0.2cm}
\includegraphics[height=0.3\textwidth]{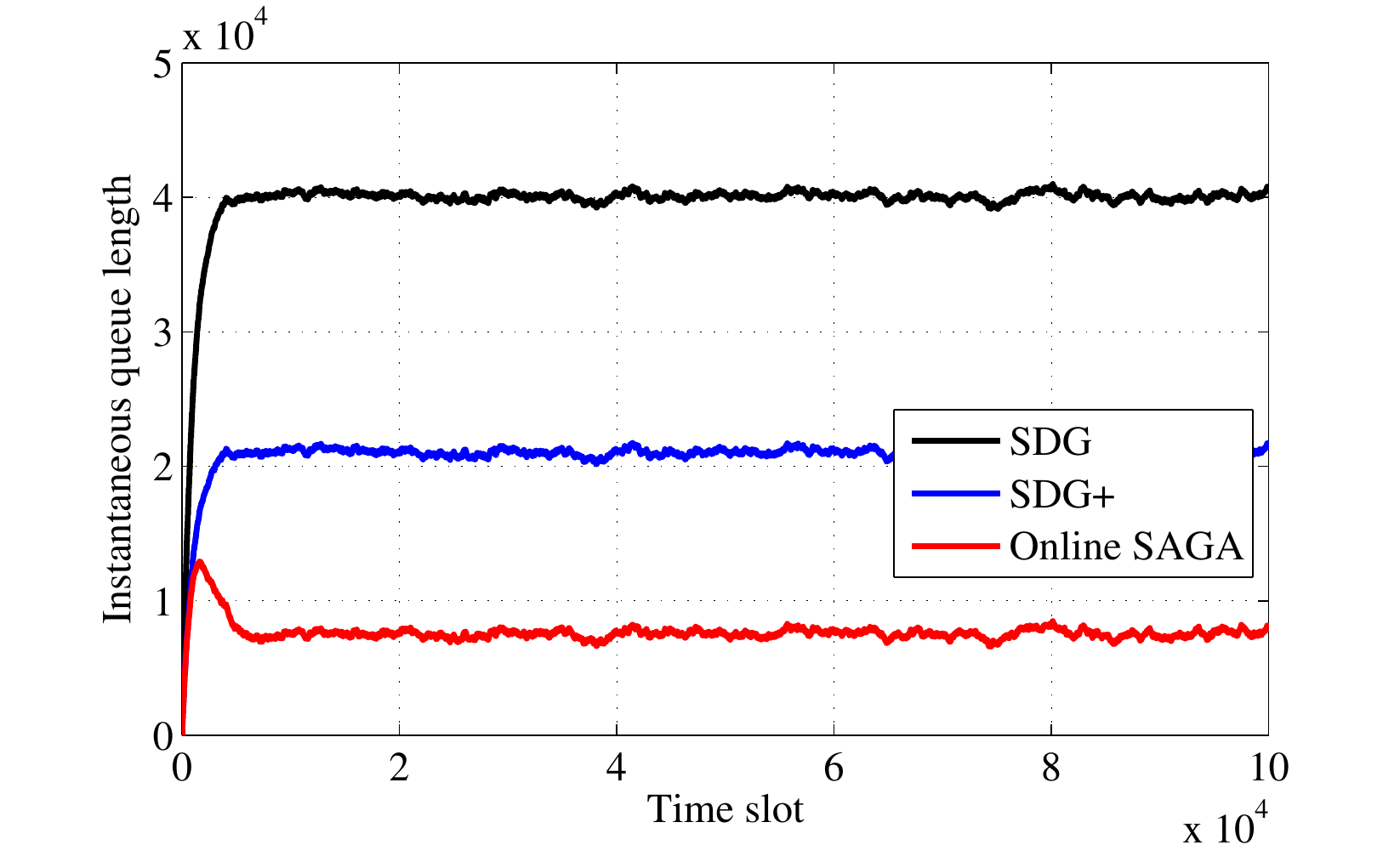}
\vspace{-0.6cm}
\caption{Comparison of queue lengths ($I=J=4$, $N_{\rm off}=1,000$, $K=2$)}
\label{Fig.DCQ}
\vspace{-0.4cm}
\end{figure}

\begin{figure}[t]
\centering
\vspace{-0.2cm}
\includegraphics[height=0.3\textwidth]{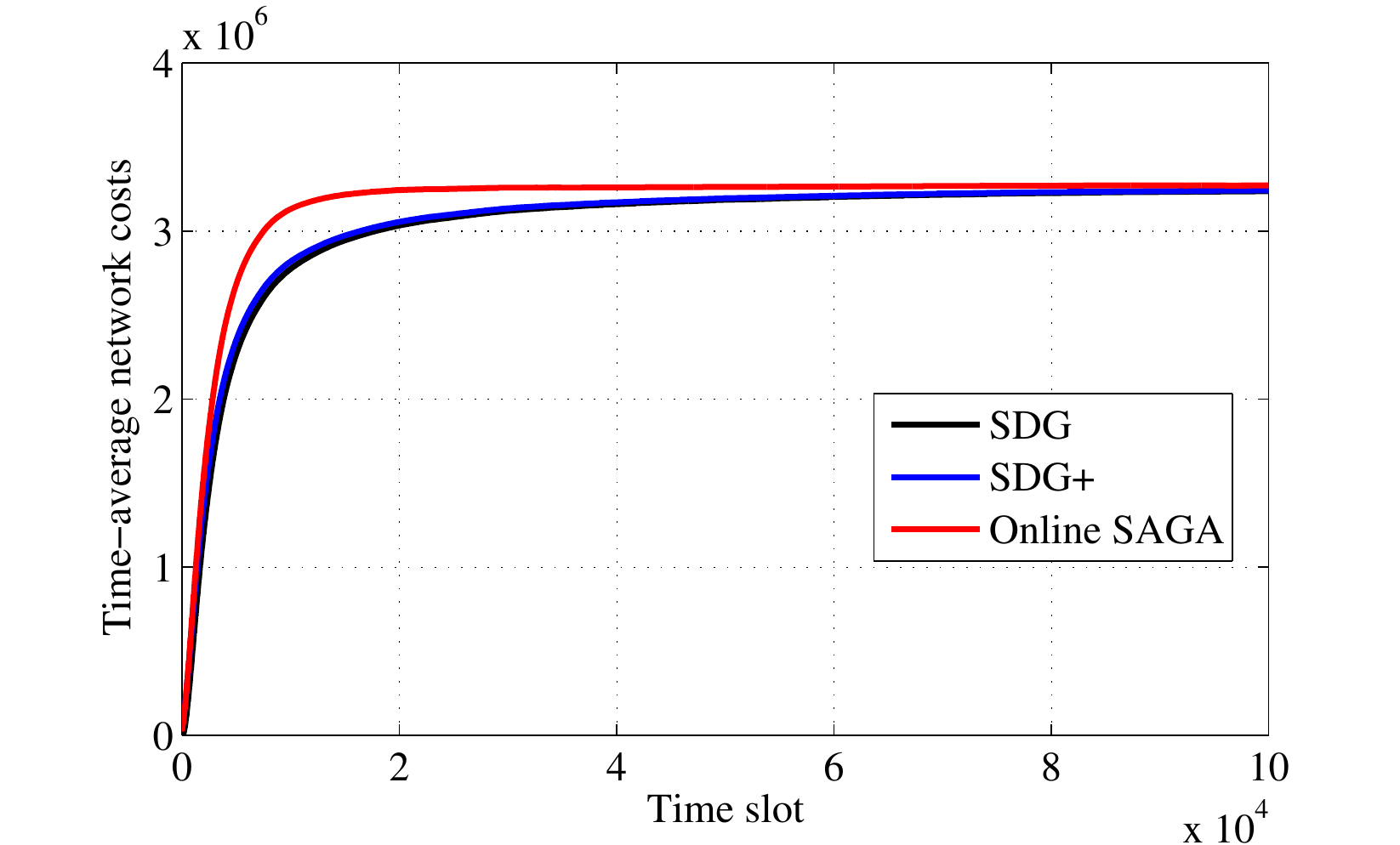}
\vspace{-0.6cm}
\caption{Comparison of time-average costs ($I\!=\!J\!=\!20$, $N_{\rm off}=1,000$, $K=2$)}
\label{Fig.obj2}
\end{figure}

\begin{figure}[t]
\centering
\vspace{-0.2cm}
\includegraphics[height=0.3\textwidth]{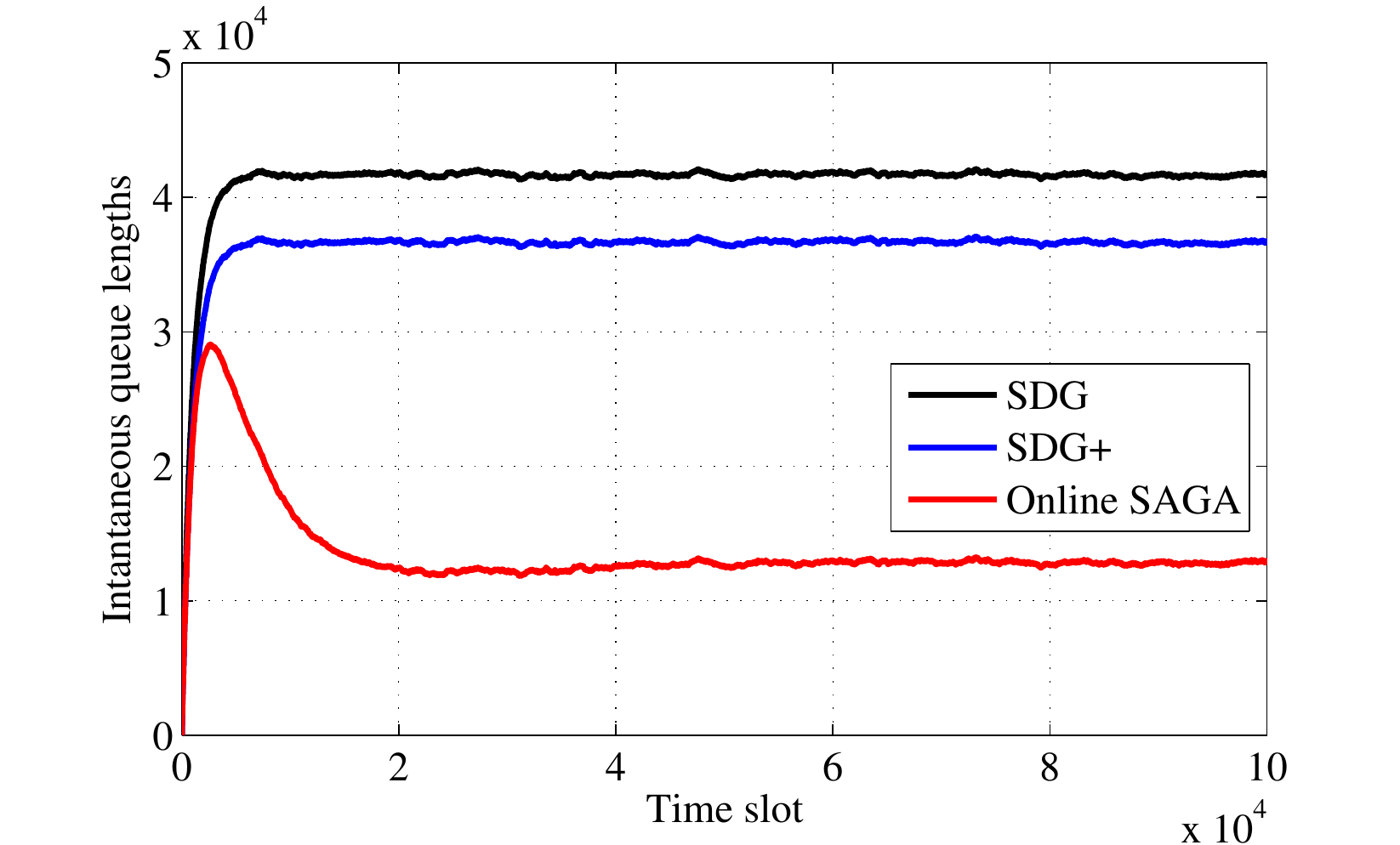}
\vspace{-0.6cm}
\caption{Comparison of queue lengths ($I=J=20$, $N_{\rm off}=1,000$, $K=2$)}
\label{Fig.DCQ2}
\vspace{-0.4cm}
\end{figure}

\begin{figure}[t]
\centering
\vspace{-0.2cm}
\includegraphics[height=0.3\textwidth]{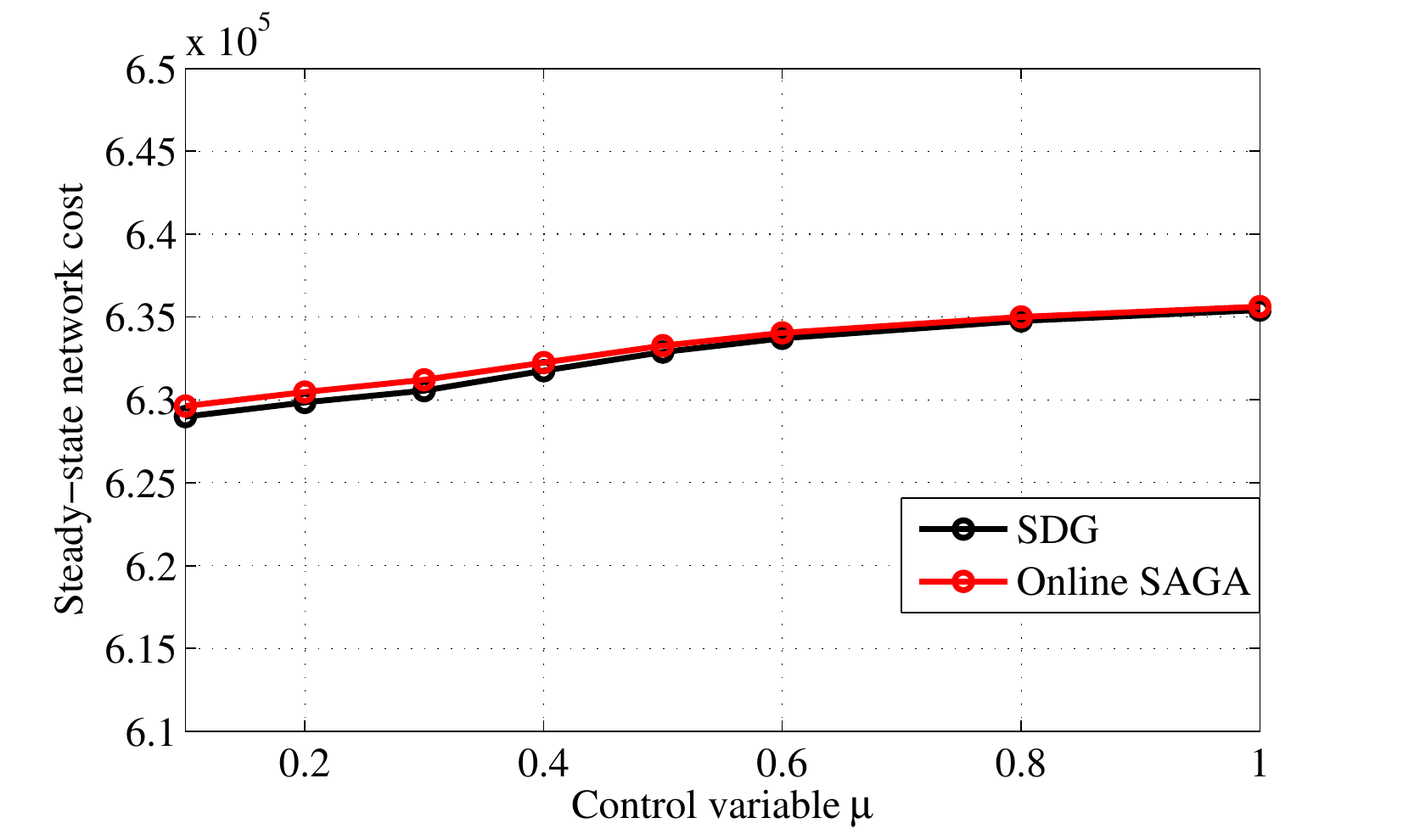}
\vspace{-0.6cm}
\caption{Network cost as a function of $\mu$ ($I=J=4$, $N_{\rm off}=0$, $K=2$)}
\label{Fig.mu2}
\end{figure}

\begin{figure}[t]
\centering
\vspace{-0.2cm}
\includegraphics[height=0.3\textwidth]{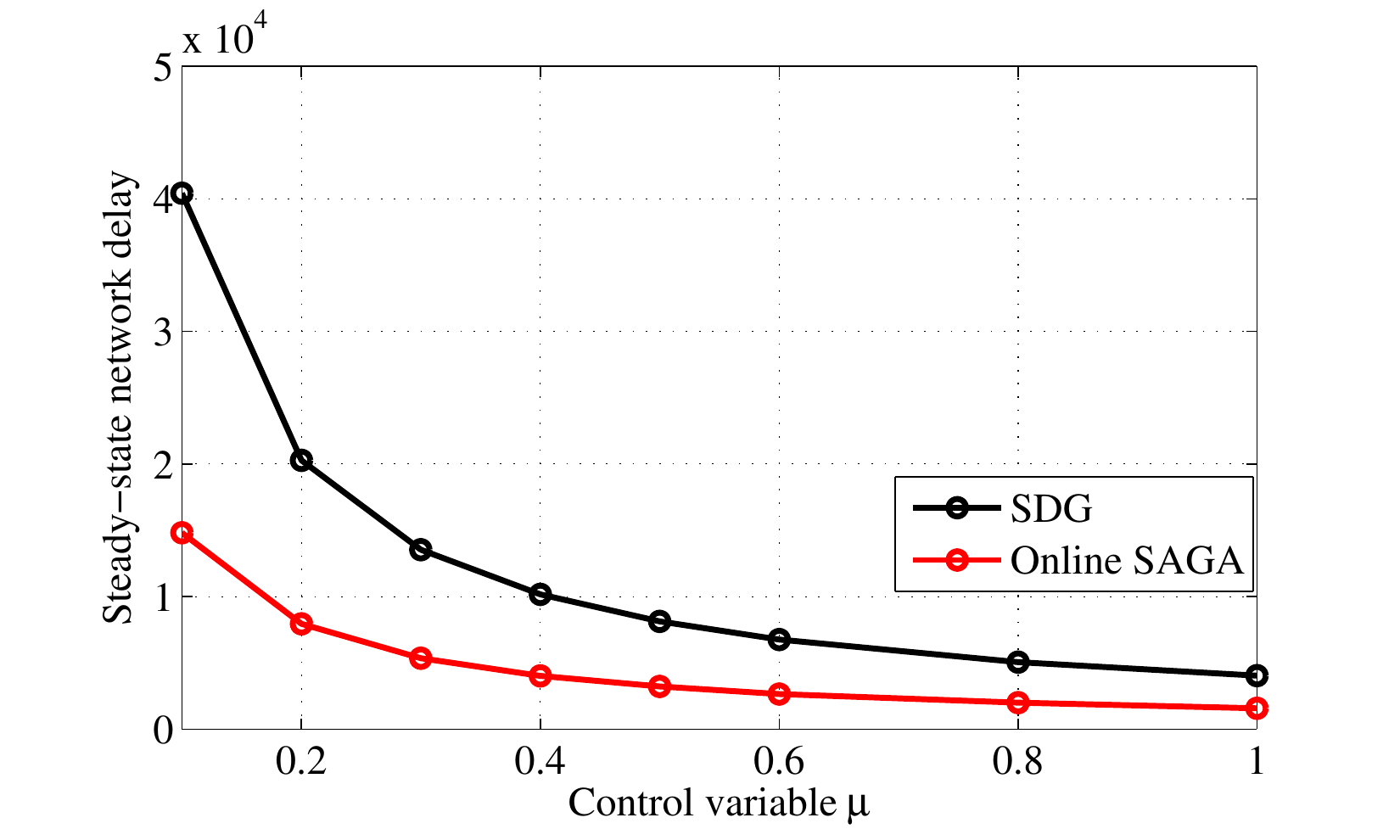}
\vspace{-0.6cm}
\caption{Network delay as a function of $\mu$ ($I=J=4$, $N_{\rm off}=0$, $K=2$)}
\label{Fig.mu1}
\vspace{-0.4cm}
\end{figure}

Performance is first compared with moderate size $N_{\rm off}=1,000$ in Figs. \ref{Fig.obj}-\ref{Fig.DCQ}. For the network cost, the three algorithms converge to the same value, with the online SAGA exhibiting fastest convergence since it quickly achieves the optimal scheduling by learning from training samples.
In addition, leveraging its learning-while-adapting attribute, online SAGA incurs considerably lower delay as the average queue size is only 40\% of that under SDG+, and 20\% of that under SDG. Clearly, offline training improves also delay performance of SDG+, and online SAGA further improves relative to SDG+ thanks to online (in addition to offline) learning.

These two metrics are further compared in Figs. \ref{Fig.obj2}-\ref{Fig.DCQ2} over a larger network with $I=20$ DCs and $J=20$ MNs. While the three algorithms exhibit similar performance in terms of network cost, the average delay of online SAGA stays noticeably lower than its alternatives. The delay performance of SDG+ approaches that of SDG (cf. the gap in Fig. \ref{Fig.DCQ}), which indicates that $N_{\rm off}=1,000$ training samples may be not enough to obtain a sufficiently accurate initialization for SDG+ over such a large network.

For fairness to SDG, the effects of tuning $\mu$ and $K$ are studied in Figs. \ref{Fig.mu2}-\ref{Fig.N2} without training ($N_{\rm off}=0$). For all choices of $\mu$, the online SAGA attains a much smaller delay at network cost comparable to that of SDG (cf. Fig. \ref{Fig.mu2}), and its delay increases much slower as $\mu$ decreases (cf. Fig. \ref{Fig.mu1}), thanks to a better delay-cost tradeoff $[\mu,\frac{1}{\sqrt{\mu}}\log^2(\mu)]$.
As shown by Lemma \ref{error-rec}, the per-slot learning error decreases as the number of learning iterations $K$ increases. Finally, Figs. \ref{Fig.N1} and \ref{Fig.N2} delineate the performance versus learning complexity tradeoff as a function of $K$. Increasing $K$ slightly improves convergence speed, and significantly reduces the average delay at the expense of higher computational complexity.

\section{Concluding Remarks}\label{sec.Cons}
A comprehensive approach was developed, which permeates benefits of stochastic averaging to considerably improve gradient updates of online dual variable iterations involved in constrained optimization tasks. The impact of integrating offline and online learning while adapting was demonstrated in a stochastic resource allocation problem emerging from cloud networks with data centers and renewables.
The considered stochastic resource allocation task was formulated with the goal of learning Lagrange multipliers in a fast and efficient manner.
Casting batch learning as maximizing the sum of finite concave functions, machine learning tools were adopted to leverage training samples for offline learning initial multiplier estimates.
The latter provided a ``hot-start'' of novel learning-while-adapting online SAGA.
The novel scheduling protocol offers low-complexity yet efficient learning, and adaptation, while guaranteeing queue stability.
Finally, performance analysis - both analytical and simulation based - demonstrated that the novel approach markedly improves network resource allocation performance, at the cost of affordable extra memory to store gradient values, and just one extra sample evaluation per slot.

This novel offline-aided-online optimization framework opens up some
new research directions, which include studying how to further endow performance improvements under other system configurations (e.g., local polyhedral property), and how to establish performance guarantee in the almost sure sense.
Fertilizing other emerging cyber-physical systems (e.g., power networks) is a promising research direction, and accounting for possibly non-stationary dynamics is also of practical interest.


\begin{thebibliography}{10}
\providecommand{\url}[1]{#1}
\csname url@samestyle\endcsname
\providecommand{\newblock}{\relax}
\providecommand{\bibinfo}[2]{#2}
\providecommand{\BIBentrySTDinterwordspacing}{\spaceskip=0pt\relax}
\providecommand{\BIBentryALTinterwordstretchfactor}{4}
\providecommand{\BIBentryALTinterwordspacing}{\spaceskip=\fontdimen2\font plus
\BIBentryALTinterwordstretchfactor\fontdimen3\font minus
  \fontdimen4\font\relax}
\providecommand{\BIBforeignlanguage}[2]{{%
\expandafter\ifx\csname l@#1\endcsname\relax
\typeout{** WARNING: IEEEtran.bst: No hyphenation pattern has been}%
\typeout{** loaded for the language `#1'. Using the pattern for}%
\typeout{** the default language instead.}%
\else
\language=\csname l@#1\endcsname
\fi
#2}}
\providecommand{\BIBdecl}{\relax}
\BIBdecl

\bibitem{chen2016b}
T.~Chen, A.~Mokhtari, X.~Wang, A.~Ribeiro, and G.~B. Giannakis, ``A data-driven
  approach to stochastic network optimization,'' in \emph{Proc. {IEEE} Global
  Conf. on Signal and Info. Process.}, Washington, DC, Dec. 2016.

\bibitem{Data2015}
\BIBentryALTinterwordspacing
J.~Whitney and P.~Delforge, ``Data center efficiency assessment,'' \emph{Issue
  Paper}, 2015. [Online]. Available:
  \url{http://www.nrdc.org/energy/data-center-efficiency-assessment.asp}
\BIBentrySTDinterwordspacing

\bibitem{gao2012}
P.~X. Gao, A.~R. Curtis, B.~Wong, and S.~Keshav, ``It's not easy being green,''
  in \emph{Proc. ACM SIGCOMM}, vol.~42, no.~4, Helsinki, Finland, Aug. 2012,
  pp. 211--222.

\bibitem{guo14}
Y.~Guo, Y.~Gong, Y.~Fang, P.~P. Khargonekar, and X.~Geng, ``Energy and network
  aware workload management for sustainable data centers with thermal
  storage,'' \emph{{IEEE} Trans. Parallel and Distrib. Syst.}, vol.~25, no.~8,
  pp. 2030--2042, Aug. 2014.

\bibitem{Urg11}
R.~Urgaonkar, B.~Urgaonkar, M.~Neely, and A.~Sivasubramaniam, ``Optimal power
  cost management using stored energy in data centers,'' in \emph{Proc. ACM
  SIGMETRICS}, San Jose, CA, Jun. 2011, pp. 221--232.

\bibitem{chen2016}
T.~Chen, X.~Wang, and G.~B. Giannakis, ``Cooling-aware energy and workload
  management in data centers via stochastic optimization,'' \emph{{IEEE} J.
  Sel. Topics Signal Process.}, vol.~10, no.~2, pp. 402--415, Mar. 2016.

\bibitem{Yao12data}
Y.~Yao, L.~Huang, A.~Sharma, L.~Golubchik, and M.~Neely, ``Data centers power
  reduction: A two time scale approach for delay tolerant workloads,'' in
  \emph{Proc. IEEE INFOCOM}, Orlando, FL, Mar. 2012, pp. 1431--1439.

\bibitem{chen2016jsac}
T.~Chen, Y.~Zhang, X.~Wang, and G.~B. Giannakis, ``Robust workload and energy
  management for sustainable data centers,'' \emph{IEEE J. Sel. Areas Commun.},
  vol.~34, no.~3, pp. 651--664, Mar. 2016.

\bibitem{liu2016}
J.~Liu, A.~Eryilmaz, N.~B. Shroff, and E.~S. Bentley, ``Heavy-ball: A new
  approach to tame delay and convergence in wireless network optimization,'' in
  \emph{Proc. IEEE INFOCOM}, San Francisco, CA, Apr. 2016.

\bibitem{li2015}
B.~Li, R.~Li, and A.~Eryilmaz, ``On the optimal convergence speed of wireless
  scheduling for fair resource allocation,'' \emph{{IEEE/ACM} Trans.
  Networking}, vol.~23, no.~2, pp. 631--643, Apr. 2015.

\bibitem{zargham2013}
M.~Zargham, A.~Ribeiro, and A.~Jadbabaie, ``Accelerated backpressure
  algorithm,'' \emph{arXiv preprint:1302.1475}, Feb. 2013.

\bibitem{ribeiro2010}
A.~Ribeiro, ``Ergodic stochastic optimization algorithms for wireless
  communication and networking,'' \emph{{IEEE} Trans. Signal Processing},
  vol.~58, no.~12, pp. 6369--6386, Jul. 2010.

\bibitem{huang2011}
L.~Huang and M.~J. Neely, ``Delay reduction via {L}agrange multipliers in
  stochastic network optimization,'' \emph{{IEEE} Trans. Automat. Contr.},
  vol.~56, no.~4, pp. 842--857, Apr. 2011.

\bibitem{huang2013}
------, ``Utility optimal scheduling in energy-harvesting networks,''
  \emph{{IEEE/ACM} Trans. Networking}, vol.~21, no.~4, pp. 1117--1130, Aug.
  2013.

\bibitem{valls2015}
V.~Valls and D.~J. Leith, ``Descent with approximate multipliers is enough:
  Generalising max-weight,'' \emph{arXiv preprint:1511.02517}, Nov. 2015.

\bibitem{vapnik2013}
V.~Vapnik, \emph{The Nature of Statistical Learning Theory}.\hskip 1em plus
  0.5em minus 0.4em\relax Berlin, Germany: Springer Science \& Business Media,
  2013.

\bibitem{huang2014}
L.~Huang, X.~Liu, and X.~Hao, ``The power of online learning in stochastic
  network optimization,'' in \emph{Proc. ACM SIGMETRICS}, vol.~42, no.~1, New
  York, NY, Jun. 2014, pp. 153--165.

\bibitem{zhang2016}
T.~Zhang, H.~Wu, X.~Liu, and L.~Huang, ``Learning-aided scheduling for mobile
  virtual network operators with {QoS} constraints,'' in \emph{Proc. WiOpt},
  Tempe, AZ, May 2016.

\bibitem{defazio2014}
A.~Defazio, F.~Bach, and S.~Lacoste-Julien, ``{SAGA}: {A} fast incremental
  gradient method with support for non-strongly convex composite objectives,''
  in \emph{Proc. Advances in Neural Info. Process. Syst.}, Montreal, Canada,
  Dec. 2014, pp. 1646--1654.

\bibitem{Daneshmand16}
H.~Daneshmand, A.~Lucchi, and T.~Hofmann, ``Starting small - learning with
  adaptive sample sizes,'' in \emph{Proc. Intl. Conf. on Mach. Learn.}, New
  York, NJ, Jun. 2016, pp. 1463--1471.

\bibitem{neely2010}
M.~J. Neely, ``Stochastic network optimization with application to
  communication and queueing systems,'' \emph{Synthesis Lectures on
  Communication Networks}, vol.~3, no.~1, pp. 1--211, 2010.

\bibitem{xu2013}
H.~Xu and B.~Li, ``Joint request mapping and response routing for
  geo-distributed cloud services,'' in \emph{Proc. IEEE INFOCOM}, Turin, Italy,
  Apr. 2013, pp. 854--862.

\bibitem{sun2016}
S.~Sun, M.~Dong, and B.~Liang, ``Distributed real-time power balancing in
  renewable-integrated power grids with storage and flexible loads,''
  \emph{{IEEE} Trans. Smart Grid}, vol.~7, no.~5, pp. 2337--2349, Sep. 2016.

\bibitem{gatsis2010}
N.~Gatsis, A.~Ribeiro, and G.~B. Giannakis, ``A class of convergent algorithms
  for resource allocation in wireless fading networks,'' \emph{{IEEE} Trans.
  Wireless Commun.}, vol.~9, no.~5, pp. 1808--1823, May 2010.

\bibitem{gregoire2015}
J.~Gregoire, X.~Qian, E.~Frazzoli, A.~de~La~Fortelle, and T.~Wongpiromsarn,
  ``Capacity-aware backpressure traffic signal control,'' \emph{{IEEE} Trans.
  Control of Network Systems}, vol.~2, no.~2, pp. 164--173, June 2015.

\bibitem{bertsekas1999}
D.~P. Bertsekas, \emph{Nonlinear Programming}.\hskip 1em plus 0.5em minus
  0.4em\relax Belmont, MA: Athena scientific, 1999.

\bibitem{robbins1951}
H.~Robbins and S.~Monro, ``A stochastic approximation method,'' \emph{Annals of
  Mathematical Statistics}, vol.~22, no.~3, pp. 400--407, Sep. 1951.

\bibitem{kong1995}
V.~Kong and X.~Solo, \emph{Adaptive {S}ignal {P}rocessing {A}lgorithms}.\hskip
  1em plus 0.5em minus 0.4em\relax Upper Saddle River, NJ: Prentice Hall, 1995.

\bibitem{chen2017tpds}
T.~Chen, A.~G. Marques, and G.~B. Giannakis, ``{DGLB}: Distributed stochastic
  geographical load balancing over cloud networks,'' \emph{{IEEE} Trans.
  Parallel and Distrib. Syst.}, to appear, 2017.

\bibitem{eryilmaz2006}
A.~Eryilmaz and R.~Srikant, ``Joint congestion control, routing, and {MAC} for
  stability and fairness in wireless networks,'' \emph{IEEE J. Sel. Areas
  Commun.}, vol.~24, no.~8, pp. 1514--1524, Aug. 2006.

\bibitem{Geor06}
L.~Georgiadis, M.~Neely, and L.~Tassiulas, ``Resource allocation and
  cross-layer control in wireless networks,'' \emph{Found. and Trends in
  Networking}, vol.~1, pp. 1--144, 2006.

\bibitem{roux2012}
N.~L. Roux, M.~Schmidt, and F.~R. Bach, ``A stochastic gradient method with an
  exponential convergence rate for finite training sets,'' in \emph{Proc.
  Advances in Neural Info. Process. Syst.}, Dec. 2012, pp. 2663--2671.

\bibitem{bousquet2008}
O.~Bousquet and L.~Bottou, ``The tradeoffs of large scale learning,'' in
  \emph{Proc. Advances in Neural Info. Process. Syst.}, Dec. 2008, pp.
  161--168.

\bibitem{mokhtari2016}
A.~Mokhtari, H.~Daneshmand, A.~Lucchi, T.~Hofmann, and A.~Ribeiro, ``Adaptive
  {N}ewton method for empirical risk minimization to statistical accuracy,'' in
  \emph{Proc. Advances in Neural Info. Process. Syst.}, Barcelona, Spain, Dec.
  2016, pp. 4062--4070.

\bibitem{yu2016}
H.~Yu and M.~J. Neely, ``A low complexity algorithm with $
  \mathcal{O}(\sqrt{T}) $ regret and constraint violations for online convex
  optimization with long term constraints,'' \emph{arXiv preprint:1604.02218},
  Apr. 2016.

\bibitem{koshal2011}
J.~Koshal, A.~Nedic, and U.~V. Shanbhag, ``Multiuser optimization: distributed
  algorithms and error analysis,'' \emph{SIAM J. Optimization}, vol.~21, no.~3,
  pp. 1046--1081, Jul. 2011.

\bibitem{meyn2012}
S.~P. Meyn and R.~L. Tweedie, \emph{Markov Chains and Stochastic
  Stability}.\hskip 1em plus 0.5em minus 0.4em\relax Berlin, Germany: Springer
  Science \& Business Media, 2012.

\end{thebibliography}

\vspace{-0.2cm}
\appendix
\subsection{Sub-optimality of the regularized dual problem}
We first state the sub-optimality brought by subtracting an $\ell_2$-regularizer in \eqref{eq.dual-func}.
\begin{lemma}\label{error-reg}
Let $\{\mathbf{x}_t^*\}$ and $\bm{\lambda}^*$ denote the optimal arguments of the primal and dual problems in \eqref{eq.reform2} and \eqref{eq.dual-prob}, respectively; and correspondingly $\{\hat{\mathbf{x}}_t^*\}$ and $\hat{\bm{\lambda}}^*$ for the primal and the $\ell_2$-regularized dual problem
\begin{align}\label{eq.regdual-prob}
 \max_{\bm{\lambda}\geq \mathbf{0}} \, {\cal D}(\bm{\lambda})-\frac{\epsilon}{2}\|\bm{\lambda}\|^2.
\end{align}
For $\epsilon>0$, it then holds under {(as1)} and {(as2)} that
	\begin{equation}
		\mathbb{E}[\|\mathbf{x}_t^*-\hat{\mathbf{x}}_t^*\|^2]\leq \frac{\epsilon}{2\sigma}\left(\|\bm{\lambda}^*\|^2-\|\hat{\bm{\lambda}}^*\|^2\right),~\forall~\bm{\lambda}^*\in \bm{\Lambda}^*
	\end{equation}
	where $\bm{\Lambda}^*$ denotes the set of optimal dual variables for the original dual problem \eqref{eq.dual-prob}.
\end{lemma}
\begin{proof}
  Follows the steps in \cite[Lemma 3.2]{koshal2011}.
\end{proof}

Lemma \ref{error-reg} provides an upper bound on the expected difference between the optimal arguments of the primal \eqref{eq.reform2} and that for \eqref{eq.regdual-prob}, in terms of the difference between the corresponding dual arguments. Clearly, by choosing a relatively small $\epsilon$, the gap between $\mathbf{x}_t^*$ and $\hat{\mathbf{x}}_t^*$ closes.

Using the convexity of $\|\cdot\|^2$, Lemma \ref{error-reg} implies readily the following bound
$\mathbb{E}[\|\mathbf{x}_t^*-\hat{\mathbf{x}}_t^*\|]\leq \sqrt{\epsilon/(2\sigma)}~\max_{\bm{\lambda}^*\in \bm{\Lambda}^*}\|\bm{\lambda}^*\|$,
which in turn implies that with the regularizer present in \eqref{eq.regdual-prob}, $\{\mathbf{x}_t^*\}$ will be ${\cal O}(\sqrt{\epsilon})$-optimal and feasible. The sub-optimality in terms of the objective value can be easily quantified using the Lipschitz gradient condition in {(as2)}.
Clearly, the gap vanishes as $\epsilon \rightarrow 0$, or, as the primal strong convexity constant $\sigma\rightarrow \infty$. As we eventually pursue an ${\cal O}(\mu)$-optimal solution in Theorem \ref{gap-onlineSAGA}, it suffices to set $\epsilon={\cal O}(\mu)$.

\begin{figure}[t]
\centering
\vspace{-0.2cm}
\includegraphics[height=0.3\textwidth]{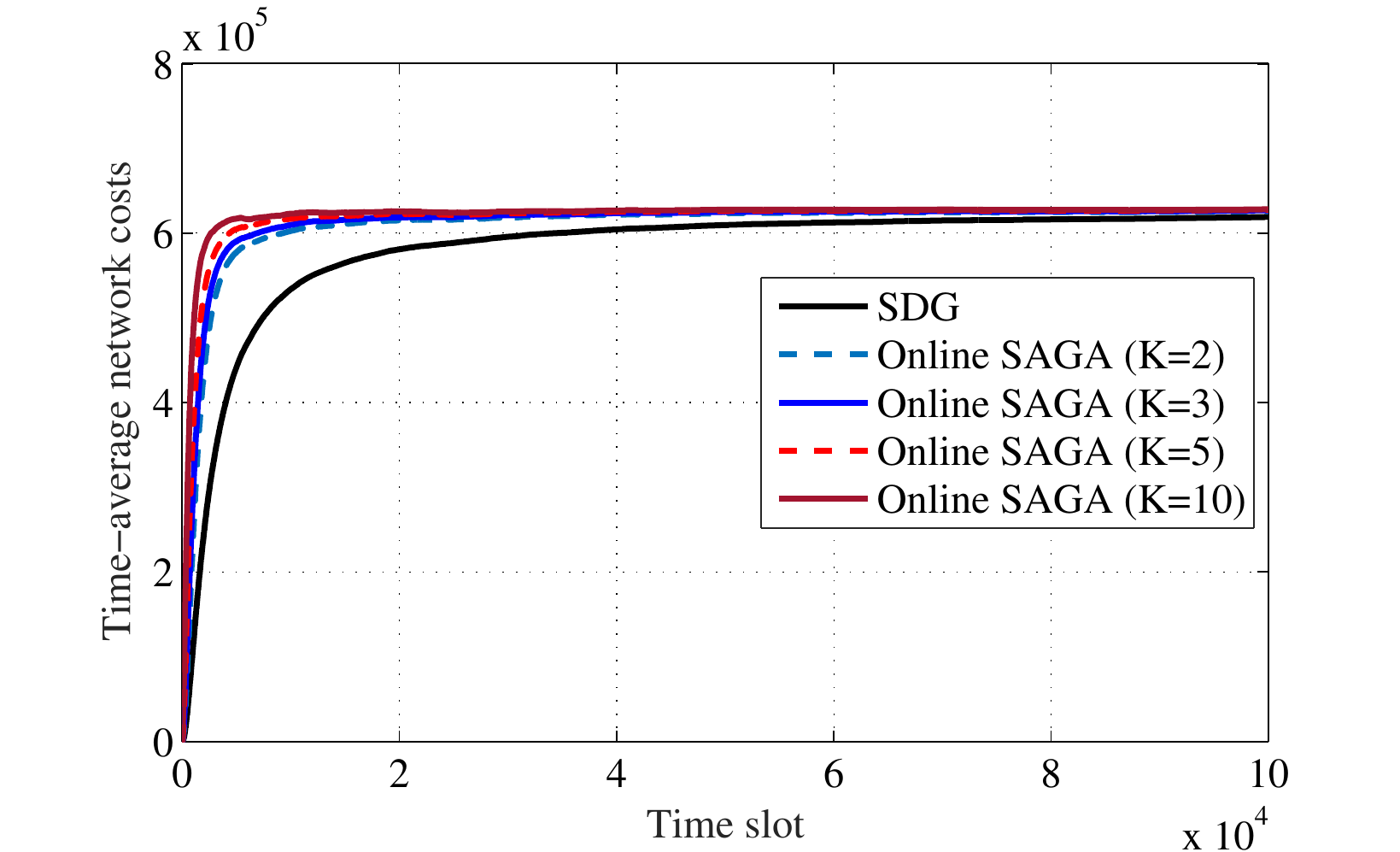}
\vspace{-0.6cm}
\caption{Network cost versus $K$ ($I=J=4$, $N_{\rm off}=0$)}
\label{Fig.N1}
\vspace{-0.2cm}
\end{figure}

\begin{figure}[t]
\centering
\vspace{-0.2cm}
\includegraphics[height=0.3\textwidth]{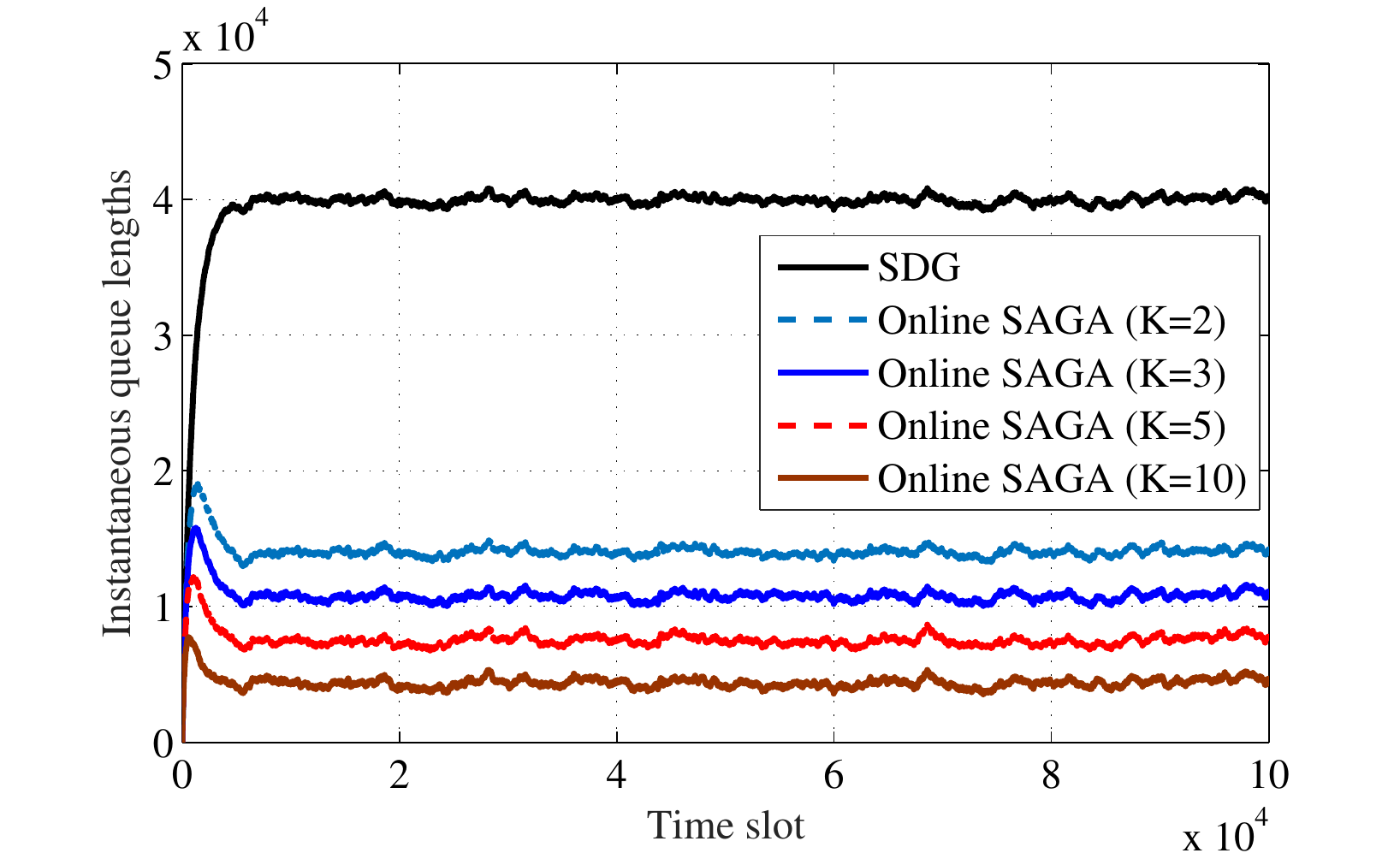}
\vspace{-0.6cm}
\caption{Queue length averaged over all nodes versus $K$ ($I\!=\!J\!=4$, $N_{\rm off}=0$)}
\label{Fig.N2}
\vspace{-0.4cm}
\end{figure}

\subsection{Proof sketch of Theorem \ref{the.linear-rate}}\label{app.linear-rate}
The SAGA in \cite{defazio2014} is originally designed for unconstrained optimization problem in machine learning applications.
Observe though, that it can achieve linear convergence rate when the objective is in a composite form $\min_{\mathbf{x}\in\mathbb{R}^d}\; f(\mathbf{x})+h(\mathbf{x})$,
where $f(\mathbf{x})$ is a strongly-convex and smooth function, and $h(\mathbf{x})$ is a convex and possibly non-smooth regularizer.
To tackle this non-smooth term, instead of pursuing a subgradient direction, the SAGA iteration in \cite{defazio2014} performs a stochastic-averaging step w.r.t. $f(\mathbf{x})$, followed by a proximal step, given by
\begin{equation}
{\rm Prox}_{\eta}^h(\mathbf{x}):=\arg\min_{\mathbf{y}\in \mathbb{R}^d}\left\{h(\mathbf{y})+ \frac{1}{2\eta}\|\mathbf{y}-\mathbf{x}\|^2\right\}.	
\end{equation}

For the constrained optimization problem \eqref{eq.dual-prob2}, we can view the regularizer as an indicator function, namely $h(\mathbf{x}):=0,\,\mathbf{x}\geq \mathbf{0};\,\infty,\,{\rm otherwise}$, which penalizes the argument outside of the feasible region $\mathbf{x}\geq \mathbf{0}$.
In this case, one can show that the proximal operator reduces to the projection operator, namely
\begin{equation}
	{\rm Prox}_{\eta}^h(\mathbf{x})=\arg\min_{\mathbf{y}\geq \mathbf{0}}\;\frac{1}{2}\|\mathbf{y}-\mathbf{x}\|^2=[\mathbf{x}]^+.
\end{equation}
To this end, the proof of linear convergence for the offline SAGA in Algorithm \ref{algo} can follow the steps of that in \cite{defazio2014}.


\subsection{Proof of Lemma \ref{error-rec}}\label{app.B}
	To formally bound the optimization error of online SAGA at each time slot, it is critical to understand the quality of using the iterate learned from a small dataset (e.g., ${\cal S}_{t-1}$) as a hot-start for the learning process in a larger set (e.g., ${\cal S}_t$).
	Intuitively, if the iterate learned from a small dataset can perform reasonably well in an ERM defined by a bigger set, then a desired learning performance of online SAGA can be expected.
	
	To achieve this goal, we first introduce a useful lemma from \cite[Theorem 3]{Daneshmand16} to relate the sub-optimality on an arbitrary set ${\cal S}$ to the sub-optimality bound on its subset $\check{\cal S}$, where $\check{\cal S}\subseteq {\cal S}$.	
	\begin{lemma}\label{small-lemma5}
		Consider the definitions $m:=|\check{\cal S}|$, $n=|{\cal S}|$, with $m<n$, and suppose that $\bm{\lambda}_{\check{\cal S}}$ is an $\omega$-optimal solution for the training subset $\check{\cal S}$ satisfying $\mathbb{E}[\hat{{\cal D}}_{\check{\cal S}}(\bm{\lambda}^*_{\check{\cal S}})-\hat{{\cal D}}_{\check{\cal S}}(\bm{\lambda}_{\check{\cal S}})]\leq \omega$. Under (as1), the sub-optimality gap when using $\bm{\lambda}_{\check{\cal S}}$ in an ERM defined by ${\cal S}$ is given by
		\begin{equation}\label{The3-small}
			\mathbb{E}\left[\hat{{\cal D}}_{\cal S}(\bm{\lambda}^*_{\cal S})-\hat{{\cal D}}_{\cal S}(\bm{\lambda}_{\check{\cal S}})\right]\leq \omega+\frac{n-m}{n}{\cal H}_s(m),
		\end{equation}
		where expectation is over all source of randomness; $\hat{{\cal D}}_{\check{\cal S}}(\bm{\lambda})$ and $\hat{{\cal D}}_{\cal S}(\bm{\lambda})$ are the empirical functions in \eqref{eq.dual-func2} approximated by sample sets $\check{\cal S}$ and ${\cal S}$, respectively; $\bm{\lambda}^*_{\check{\cal S}}$ and $\bm{\lambda}^*_{\cal S}$ are the corresponding two optimal multipliers; and, ${\cal H}_s(m)$ is the statistical error induced by the $m$ samples defined in \eqref{ineq.gene}.
	\end{lemma}
	Lemma \ref{small-lemma5} states that if the set $\check{\cal S}$ grows to ${\cal S}$, then the optimization error on the larger set ${\cal S}$ is bounded by the original error $\omega$ plus an additional ``growth cost'' $(1-m/n){\cal H}_s(m)$.
	Note that Lemma \ref{small-lemma5} considers a general case of increasing training set from $\check{\cal S}$ to ${\cal S}$, but the dynamic learning steps of online SAGA in Algorithm \ref{algo2} is a special case when $\check{\cal S}\!=\!{\cal S}_{t-1}$ and ${\cal S}\!=\!{\cal S}_t$.

Building upon Lemma \ref{small-lemma5} and Theorem \ref{the.linear-rate}, we are ready to derive the optimization error of online SAGA.
To achieve this goal, we first define a sequence ${\cal U}_o(k;n)$, which will prove to provide an upper-bound for the optimization error by \textit{incrementally} running $k$ offline SAGA iterations within a dataset ${\cal S}$ with $n=|{\cal S}|$; i.e., $\mathbb{E}[\hat{\cal D}_{\cal S}(\bm{\lambda}^*_{\cal S})-\hat{\cal D}_{\cal S}(\bm{\lambda}_k)]\leq {\cal U}_o(k;n)$.
Note that index $k$ is used to differentiate from time index $t$, since it includes the SAGA iterations in both offline and online phases.

Considering a dataset ${\cal S}$ with $n=|{\cal S}|$, running an additional SAGA iteration \eqref{eq.saga} will reduce the error by a factor of $\Gamma_n$ defined in \eqref{eq.linear-fun}.
Hence, if ${\cal U}_o(k-1;n)$ gives an upper-bound of $\mathbb{E}[\hat{\cal D}_{\cal S}(\bm{\lambda}^*_{\cal S})-\hat{\cal D}_{\cal S}(\bm{\lambda}_{k-1})]$, it follows that
\begin{equation}\label{eq.onestep-1}
\mathbb{E}[\hat{\cal D}_{\cal S}(\bm{\lambda}^*_{\cal S})-\hat{\cal D}_{\cal S}(\bm{\lambda}_k)]\leq \Gamma_n{\cal U}_o(k-1;n).
\end{equation}

On the other hand, the error of running $k$ iterations within $n$ data should also satisfy the inequality \eqref{The3-small} in Lemma \ref{small-lemma5} which holds from any $m<n$.
Thus, if given ${\cal U}_o(k;m)$, it follows that
\begin{equation}\label{eq.onestep-2}
\!\mathbb{E}\!\left[\hat{\cal D}_{\cal S}(\bm{\lambda}^*_{\cal S})\!-\!\hat{\cal D}_{\cal S}(\bm{\lambda}_k)\right]\!\!\leq\! \min_{m<n}\!\left[{\cal U}_o(k;m)\!+\!\frac{n-m}{n}{\cal H}_s(m)\right]\!\!\!\!
\end{equation}
 where the minimum follows from that \eqref{The3-small} holds for any $m<n$.

Combining \eqref{eq.onestep-1} and \eqref{eq.onestep-2} together, we define ${\cal U}_o(k;n)$ as the minimum of these two factors, given by
\begin{subequations}\label{eq.recur-U}
	\begin{numcases}{\hspace{-0.8cm}
	{\cal U}_o(k;n)=\min}
	\Gamma_n\,{\cal U}_o(k-1;n)\label{eq.recur-U1}\\
	\min_{m<n}\left[{\cal U}_o(k;m)+\frac{n-m}{n}{\cal H}_s(m)\right]\label{eq.recur-U2}
\end{numcases}
\end{subequations}
where the initial bound is defined as ${\cal U}_o(0;N_{\rm off})\!=\!\xi\geq \mathbb{E}[C_{{\cal S}_{\rm off}}(\bm{\lambda}_0)]$ with constant $C_{{\cal S}_{\rm off}}(\bm{\lambda}_0)$ in \eqref{eq.linear-rate_2}.
Note that following the arguments for \eqref{eq.onestep-1} and \eqref{eq.onestep-2}, one can use induction to show that the sequence ${\cal U}_o(k;n)$ indeed upper bounds the expected optimization error $\mathbb{E}[\hat{\cal D}_{\cal S}(\bm{\lambda}^*_{\cal S})-\hat{\cal D}_{\cal S}(\bm{\lambda}_k)]$ with $n=|{\cal S}|$.

As online SAGA runs on average $K$ offline SAGA iterations per sample, and the dataset ${\cal S}_t$  has $N_t$ samples, with $k=KN_t$ and $n=N_t$, we want to show the following bound
\begin{equation}\label{eq.equv-err}
	{\cal U}_o(KN_t;N_t)\leq c(K){\cal H}_s(N_t)+\frac{\xi}{e^{K/5}}\Big(\frac{3\kappa}{4N_t}\Big)^{K/5}.
\end{equation}
We prove the inequality \eqref{eq.equv-err} thus \eqref{ineq.error-rec} in Lemma \ref{error-rec} by induction.
Starting from $N_0=N_{\rm off}=3\kappa/4$ at time $t=0$, the inequality \eqref{eq.equv-err} holds since
\begin{align}
	{\cal U}_o(KN_{\rm off};N_{\rm off})\leq & (1-1/(4N_{\rm off}))^{KN_{\rm off}}{\cal U}_o(0;N_{\rm off}) \nonumber\\
	\stackrel{(a)}{\leq}& \frac{\xi}{e^{K/4}}\leq \frac{\xi}{e^{K/5}}
\end{align}
 where the inequality (a) is due to $(1-1/n)^{n}\leq 1/e$.
 Assume the inequality \eqref{eq.equv-err} holds for time $t$, and we next show that it is also true for time $t+1$ with $KN_{t+1}$ iterations and $N_{t+1}$ samples.
Specifically, we have that (cf. $N_{t+1}:=N_t+1$)
\begin{align}\label{eq.lemma1-step1}
 	&{\cal U}_o(KN_{t+1};N_{t+1}) \stackrel{(b)}{\leq}(\Gamma_{N_{t+1}})^K{\cal U}_o(KN_t;N_{t+1})\\
 	&\!\!\!\stackrel{(c)}{\leq}\left(\frac{4N_t\!+\!3}{4N_{t+1}}\right)^{\!\!K}\!\left[{\cal U}_o(KN_t;N_t)+\frac{{\cal H}_s(N_t)}{N_{t+1}}\right]\nonumber\\
 	&\!\!\!\stackrel{(d)}{\leq}\!\!\left(\frac{4N_t\!+\!3}{4N_{t+1}}\right)^{\!\!K}\!\!\left[c(K){\cal H}_s(N_t)\!+\!\frac{\xi}{e^{K/5}}\Big(\frac{3\kappa }{4N_t}\Big)^{\!K/5}\!\!\!\!+\!\frac{{\cal H}_s(N_t)}{N_{t+1}}\right]\! \nonumber
 \end{align}
 where inequality (b) uses the recursion \eqref{eq.recur-U1} for $K$ times; (c) follows from the definition of $\Gamma_{N_{t+1}}$ in \eqref{eq.linear-fun} as well as the recursion \eqref{eq.recur-U2}; and (d) holds since we assume that the argument holds for $N_t$.
 Rearranging \eqref{eq.lemma1-step1}, it follows that
 \begin{align}\label{eq.lemma1-step2}
 	\!&{\cal U}_o(KN_{t+1};N_{t+1})\leq\!\!\\
 	&\!\!\!\left(\!\frac{4N_t\!+\!3}{4N_{t+1}}\right)^{\!\!K}\!\!\!\left[c(K)\!+\!\frac{1}{N_{t+1}}\right]\!\!{\cal H}_s(N_t)\!+\!\frac{\xi}{e^{K/5}}\!\!\!\left(\frac{4N_t\!+\!3}{4N_{t+1}}\right)^{\!\!K}\!\!\!\Big(\frac{3\kappa}{4N_t}\Big)^{\!K/5}\!\!\!\!.\nonumber
 \end{align}

Next we prove the following two inequalities:
 \begin{equation*}
 \!	{\rm (e1)}\left(\frac{4N_t\!+\!3}{4N_{t+1}}\right)^{\!5}\!\cdot\frac{{\cal H}_s(N_t)}{{\cal H}_s(N_{t+1})}\leq 1;~ {\rm (e2)}\left(\!\frac{4N_t\!+\!3}{4N_{t+1}^{4/5}N_t^{1/5}}\!\right)^{\!\!K}\!\!\!\leq 1,\,\forall K
\end{equation*}
because if (e1) and (e2) hold, one can upper bound the two terms in the RHS of \eqref{eq.lemma1-step2} separately and arrive at
 \begin{align}\label{eq.lemma1-54}
 &{\cal U}_o(KN_{t+1};N_{t+1})\leq\!\!\\
 &\left(\frac{4N_t+3}{4N_{t+1}}\right)^{\!\!K\!-5}\!\!\Big[c(K)\!+\!\frac{1}{N_{t+1}}\!\Big]{\cal H}_s(N_{t+1})\!+\!\frac{\xi}{e^{K/5}}\!\!\left(\!\frac{3\kappa}{4N_{t+1}}\!\right)^{\!\!K/5}\!\!\!\!.\nonumber
 \end{align}

 To prove (e1) and (e2), we repeatedly use Bernoulli's inequality, namely, $\left(1+x/p\right)^p\geq \left(1+x/q\right)^q,\;\forall x>0,\,p>q>0$, which for (e1) implies that
 \begin{equation}\label{ineq.d1}
 	\!\!\left(\frac{4N_{t+1}}{4N_t+3}\right)^{\!\!5}\!=\!\left(\!1\!+\!\frac{5}{5(4N_t+3)}\!\right)^{\!\!5}\!\stackrel{(f)}{\geq}\! 1+\!\frac{5}{4N_t+3}\stackrel{(g)}{\geq} \frac{N_{t+1}}{N_t}\!\!\!
 \end{equation}
 where (f) uses Bernoulli's inequality with $p=5$ and $q=1$, and (g) holds for $N_t\geq 3$.
 Building upon \eqref{ineq.d1}, the inequality (d1) follows since
 \begin{equation}
 	\frac{N_t}{N_{t+1}}\cdot\frac{{\cal H}_s(N_t)}{{\cal H}_s(N_{t+1})}=\left(\frac{N_t}{N_{t+1}}\right)^{\!\!1-\beta}\!\!\!\leq 1,~~~\forall \beta\in [0,1].
 \end{equation}

For (e2), since $\Big(\frac{4N_t+3}{4N_{t+1}^{4/5}N_t^{1/5}}\Big)^{\!K}\!=\!\Big(\frac{(N_t+3/4)^5}{N_t(N_t+1)^4}\Big)^{K/5}$, it suffices to show that $\frac{(N_t+3/4)^5}{N_t(N_t+1)^4}\leq 1$, which implies that
\begin{align*}
	\frac{(N_t+3/4)^5}{N_t(N_t+1)^4}&=\frac{N_t+3/4}{N_t\left(1+\frac{1/4}{N_t+3/4}\right)^4}\stackrel{(h)}{\leq} \frac{N_t+3/4}{N_t\left(1+\frac{1}{N_t+3/4}\right)}\\
	&=\frac{(N_t+3/4)^2}{N_t\left(N_t+7/4\right)}=1-\frac{N_t/4-9/16}{N_t\left(N_t+7/4\right)}\stackrel{(i)}{\leq}1
\end{align*}
 where (h) uses Bernoulli's inequality with $p=4$ and $q=1$, and (i) again holds for $N_t\geq N_{\rm off}\geq 3$, which is typically satisfied in batch training scenarios. Hence, inequality \eqref{eq.lemma1-54} holds.

Building upon \eqref{eq.lemma1-54}, we will show next that
\begin{equation}\label{eq.ck}
	\left(\frac{4N_t\!+\!3}{4N_{t+1}}\right)^{K-5}\!\left(c(K)+\frac{1}{N_{t+1}}\right)\leq c(K)
\end{equation}
from which we can conclude that \eqref{eq.equv-err} holds for $t+1$, namely
 \begin{align}\label{eq.lemma1-step3}
\!\!\!{\cal U}_o(KN_{t+1};N_{t+1})\!\leq\! c(K){\cal H}_s(N_{t+1})\!+\!\frac{\xi}{e^{K/5}}\!\left(\frac{3\kappa}{4N_{t+1}}\right)^{\!\!K/5}\!\!\!\!\!\!.\!
 \end{align}

By rearranging terms, \eqref{eq.ck} is equivalent to
\begin{align}\label{eq.lemma1-step4}
	c(K) &\geq \left(\left(\left(\frac{4N_{t+1}}{4N_t\!+\!3}\right)^{K-5}-1\right)N_{t+1}\right)^{-1}\nonumber\\
	&\stackrel{(j)}{=}\left(\left(\sum_{k=0}^{K-6}\left(\frac{4N_{t+1}}{4N_t\!+\!3}\right)^k\right)\left(\frac{4N_{t+1}}{4N_t\!+\!3}-1\right)N_{t+1}\right)^{-1}\nonumber\\
	&= {\left( \frac{1}{4} \sum_{k=1}^{K-5}\left(\frac{4N_{t+1}}{4N_t\!+\!3}\right)^k\right)}^{-1}
\end{align}
where (j) follows from the sum of a geometric sequence.
Observe that the RHS of \eqref{eq.lemma1-step4} is increasing with $t$, thus we have that
\begin{align}
c(K):=\frac{4}{K-5}&=\lim_{t\rightarrow \infty}{\left( \frac{1}{4}\sum_{k=1}^{K-4}\Big(\frac{4N_{t+1}}{4N_t+3}\Big)^k\right)}^{\!\!-1}\nonumber\\
&\geq {\left( \frac{1}{4}\sum_{k=1}^{K-4}\Big(\frac{4N_{t+1}}{4N_t+3}\Big)^k\right)}^{\!\!-1}\!,\;\forall t
\end{align}
from which \eqref{eq.ck} holds and so does inequality \eqref{eq.lemma1-step3}.
This completes the proof of Lemma \ref{error-rec}.

 \subsection{Proof of Proposition \ref{error-K=1}}\label{app.prop2}

 Using Lemma \ref{error-rec} and selecting $K=8$, we have that (cf. the definition of ${\cal U}_o(k;n)$ in Appendix \ref{app.B})
 \begin{align}
 	{\cal U}_o(8N_t;N_t)&\leq \frac{4}{3}{\cal H}_s(N_t)+\xi \Big(\frac{3\kappa}{4eN_t}\Big)^{1.6}.
 \end{align}
It further follows from Lemma \ref{small-lemma5} that
  \begin{align}\label{eq.U_rec}
  	 	&{\cal U}_o(8N_t;2N_t)\stackrel{(a)}{\leq}{\cal U}_o(8N_t;N_t)+\frac{1}{2}{\cal H}_s(N_t)\\
  	 	&\leq \frac{11}{6}{\cal H}_s(N_t)+\xi \Big(\frac{3\kappa}{4eN_t}\Big)^{1.6}\!\stackrel{(b)}{=}\!\frac{11}{6}\cdot 2^{\beta}{\cal H}_s(2N_t)\!+\!\xi \Big(\frac{3\kappa}{4eN_t}\Big)^{1.6}\nonumber
  \end{align}
where (a) follows from \eqref{eq.recur-U2}, and (b) uses the definition of ${\cal H}_s(N)$ such that ${\cal H}_s(2N_t)=2^{-\beta}{\cal H}_s(N_t)$. By setting $N_t=N_t/2$, we have
\begin{equation}
  	 	{\cal U}_o(4N_t;N_t)\leq \frac{11}{6}\cdot 2^{\beta}{\cal H}_s(N_t)+\xi \Big(\frac{3\kappa}{2eN_t}\Big)^{1.6}.
  \end{equation}
Following the steps in \eqref{eq.U_rec}, we can also deduce that
 \begin{align}
  	 	\!\!{\cal U}_o(2N_t;N_t)\!\leq\! \left(\frac{11}{6}\cdot 4^{\beta}\!+\!\frac{1}{2}\cdot 2^{\beta}\right){\cal H}_s(N_t)\!+\!\xi \Big(\frac{3\kappa}{eN_t}\Big)^{1.6}
  \end{align}
  and likewise we have
   \begin{align}
  	 	\!\!\!{\cal U}_o(N_t;N_t)\!\leq \!\left(\frac{11\times 8^{\beta}}{6}\!+\!\frac{4^{\beta}}{2}\!+\!\frac{2^{\beta}}{2}\right){\cal H}_s(N_t)\!+\!\xi \Big(\frac{6\kappa}{eN_t}\Big)^{1.6}\!\!\!
  \end{align}
which completes the proof.

\subsection{Proof of Lemma \ref{lem.drift}}\label{app.lemma2}
To prove Lemma \ref{lem.drift}, we first state a simple but useful property of the primal-dual problems \eqref{eq.reform2} and \eqref{eq.dual-prob}, which can be easily derived based on the KKT conditions \cite{bertsekas1999}.
\begin{proposition}\label{prop.primal-dual}
	Under (as1)-(as3), for the constrained optimization \eqref{eq.reform2} with the optimal policy $\bm{\chi}^*(\cdot)$ and its optimal Lagrange multiplier $\bm{\lambda}^*$, it holds that $\mathbb{E}[\mathbf{A}\mathbf{x}_t^*+\mathbf{c}_t]=\mathbf{0}$ with $\mathbf{x}_t^*=\bm{\chi}^*(\mathbf{s}_t)$, and accordingly that $\nabla {\cal D}(\bm{\lambda}^*)=\mathbf{0}$.
\end{proposition}

Based on Proposition \ref{prop.primal-dual}, we are ready to prove Lemma \ref{lem.drift}.
Since $\bm{\lambda}_t$ converges to $\bm{\lambda}^*\; {\rm whp}$, and $\mathbf{b}>\mathbf{0}$, there exists a finite $T_b$ such that for $t>T_b$, we have $\|\bm{\lambda}^*-\bm{\lambda}_t\|\leq \|\mathbf{b}\|$ and thus $\tilde{\mathbf{b}}_t\geq \mathbf{0}$ by definition. Hence,
\begin{align}
	\|\mathbf{q}_{t+1}-&\tilde{\mathbf{b}}_t/\mu\|^2\stackrel{(a)}{\leq}\|\mathbf{q}_t+\mathbf{A}\mathbf{x}_t+\mathbf{c}_t-\tilde{\mathbf{b}}_t/\mu\|^2\nonumber\\
	\stackrel{(b)}{\leq} & \|\mathbf{q}_t-\tilde{\mathbf{b}}_t/\mu\|^2+M+2(\mathbf{q}_t-\tilde{\mathbf{b}}_t/\mu)^{\top}(\mathbf{A}\mathbf{x}_t+\mathbf{c}_t)\nonumber\\
	\stackrel{(c)}{=} & \|\mathbf{q}_t-\tilde{\mathbf{b}}_t/\mu\|^2+M+2\left(\frac{\bm{\gamma}_t-\bm{\lambda}^*}{\mu}\right)^{\top}(\mathbf{A}\mathbf{x}_t+\mathbf{c}_t)\nonumber\\
	\stackrel{(d)}{\leq}&  \|\mathbf{q}_t-\tilde{\mathbf{b}}_t/\mu\|^2+\frac{2}{\mu}\left({\cal D}_t(\bm{\gamma}_t)-{\cal D}_t(\bm{\lambda}^*)\right)+M\label{ineq.drift0}
	\end{align}
where (a) comes from the non-expansive property of the projection operator; (b) uses the upper bound $M:=\max_t\max_{\mathbf{x}_t\in{\cal X}}\|\mathbf{A}\mathbf{x}_t+\mathbf{c}_t\|^2$; equality (c) uses the definitions of $\tilde{\mathbf{b}}_t$ and $\bm{\lambda}_t$; and (d) follows because $\mathbf{A}\mathbf{x}_t+\mathbf{c}_t$ is a subgradient of the concave function ${\cal D}_t(\bm{\lambda})$ at $\bm{\lambda}=\bm{\gamma}_t$ (cf. \eqref{eq.real-time}).

Using the strong concavity of ${\cal D}(\bm{\lambda})$ at $\bm{\lambda}=\bm{\lambda}^*$, it holds that
\begin{align}\label{strong-dual}
	{\cal D}(\bm{\gamma}_t)\leq &{\cal D}(\bm{\lambda}^*)+\nabla {\cal D}(\bm{\lambda}^*)^{\top}(\bm{\gamma}_t-\bm{\lambda}^*)-\frac{\epsilon}{2}\|\bm{\gamma}_t-\bm{\lambda}^*\|^2\nonumber\\
	\stackrel{(e)}{\leq} &{\cal D}(\bm{\lambda}^*)-\frac{\epsilon}{2}\|\bm{\gamma}_t-\bm{\lambda}^*\|^2
\end{align}
where (e) follows from $\nabla {\cal D}(\bm{\lambda}^*)=\mathbf{0}$ (cf. Proposition \ref{prop.primal-dual}).

Then using \eqref{strong-dual}, and taking expectations on \eqref{ineq.drift0} over $\mathbf{s}_t$ conditioned on $\mathbf{q}_t$, we have
\begin{equation}\label{ineq.drift}
	\mathbb{E}\left[\|\mathbf{q}_{t+1}-\tilde{\mathbf{b}}_t/\mu\|^2\right]\leq\|\mathbf{q}_t-\tilde{\mathbf{b}}_t/\mu\|^2-\mu\epsilon\|\mathbf{q}_t-\tilde{\mathbf{b}}_t/\mu\|^2+M
\end{equation}
where we used ${\cal D}(\bm{\lambda}):=\mathbb{E}\left[{\cal D}_t(\bm{\lambda})\right]$, and $\bm{\gamma}_t-\bm{\lambda}^*=\mu\mathbf{q}_t-\tilde{\mathbf{b}}_t$.

Hence, based on \eqref{ineq.drift}, if we have
\begin{equation}\label{ineq.vieta}
	-\mu\epsilon\|\mathbf{q}_t-\tilde{\mathbf{b}}_t/\mu\|^2+M\leq -2\sqrt{\mu}\|\mathbf{q}_t-\tilde{\mathbf{b}}_t/\mu\|+\mu
\end{equation}
it holds that $\mathbb{E}\big[\|\mathbf{q}_{t+1}-\tilde{\mathbf{b}}_t/\mu\|^2\big]\!\leq\! \big(\|\mathbf{q}_t-\tilde{\mathbf{b}}_t/\mu\|-\sqrt{\mu}\big)^2$,
and by the convexity of the quadratic function, we have
\begin{equation}
	\mathbb{E}\left[\|\mathbf{q}_{t+1}-\tilde{\mathbf{b}}_t/\mu\|\right]^2\leq \left(\|\mathbf{q}_t-\tilde{\mathbf{b}}_t/\mu\|-\sqrt{\mu}\right)^2
\end{equation}
which implies \eqref{eq.drift} in the lemma.
Vieta's formulas for second-order equations ensure existence of $B=\Theta(\frac{1}{\sqrt{\mu}})$ so that \eqref{ineq.vieta} is satisfied for $\|\mathbf{q}_t-\tilde{\mathbf{b}}_t/\mu\|>B$, and thus the lemma follows.

\subsection{Proof sketch for Theorem \ref{the.queue-stable}}\label{app.theorem3}
Lemma \ref{lem.drift} assures that $\mathbf{q}_t$ always tracks a time-varying target $\tilde{\mathbf{b}}_t/\mu$. As $\lim_{t\rightarrow \infty}\tilde{\mathbf{b}}_t/\mu=\mathbf{b}/\mu \;{\rm w.h.p.}$, $\mathbf{q}_t$ will eventually track $\mathbf{b}/\mu$ and deviate by a distance $B=\Theta(\frac{1}{\sqrt{\mu}})$, which based on the Foster-Lyapunov Criterion implies that the steady state of Markov chain $\{\mathbf{q}_t\}$ exists \cite{meyn2012}. A rigorous proof of this claim  follows the lines of \cite[Theorem 1]{huang2014}, which is omitted here.

\subsection{Proof of Theorem \ref{gap-onlineSAGA}}\label{app.theorem4}

Letting $\Delta(\mathbf{q}_t)\!:=\!\frac{1}{2}(\|\mathbf{q}_{t+1}\|^2\!-\!\|\mathbf{q}_t\|^2)$ denote the Lyapunov drift, and squaring the queue update, yields
\begin{align*}
	\|\mathbf{q}_{t+1}\|^2=&\|\mathbf{q}_t\|^2+2\mathbf{q}_t^{\top}(\mathbf{A}\mathbf{x}_t+\mathbf{c}_t)+\|\mathbf{A}\mathbf{x}_t+\mathbf{c}_t\|^2\\
	\stackrel{(a)}{\leq}&\|\mathbf{q}_t\|^2+2\mathbf{q}_t^{\top}(\mathbf{A}\mathbf{x}_t+\mathbf{c}_t)+M
\end{align*}
where (a) follows from the upper bound of $\|\mathbf{A}\mathbf{x}_t+\mathbf{c}_t\|^2$. Multiplying both sides by $\mu/2$, and adding $\Psi_t(\mathbf{x}_t)$, leads to
\begin{align*}
	\mu\Delta(\mathbf{q}_t)+&\Psi_t(\mathbf{x}_t)=\Psi_t(\mathbf{x}_t)+\mu\mathbf{q}_t^{\top}(\mathbf{A}\mathbf{x}_t+\mathbf{c}_t)+{\mu M}/{2}\\
	\stackrel{(b)}{=}&\Psi_t(\mathbf{x}_t)+(\bm{\gamma}_t-\bm{\lambda}_t+\mathbf{b})^{\top}(\mathbf{A}\mathbf{x}_t+\mathbf{c}_t)+{\mu M}/{2}\\
	\stackrel{(c)}{=}&{\cal L}_t(\mathbf{x}_t,\bm{\gamma}_t)+(\mathbf{b}-\bm{\lambda}_t)^{\top}(\mathbf{A}\mathbf{x}_t+\mathbf{c}_t)+{\mu M}/{2}
\end{align*}
where (b) uses the definition of $\bm{\gamma}_t$ and (c) is the definition of the instantaneous Lagrangian. Taking expectations over $\mathbf{s}_t$ conditioned on $\mathbf{q}_t$, we arrive at
\begin{align}\label{eq.57}
	&\mu\mathbb{E}\left[\Delta(\mathbf{q}_t)\right]+\mathbb{E}\left[\Psi_t(\mathbf{x}_t)\right]\nonumber\\\stackrel{(d)}{=}&{\cal D}(\bm{\gamma}_t)+\mathbb{E}\left[(\mathbf{b}-\bm{\lambda}_t)^{\top}(\mathbf{A}\mathbf{x}_t+\mathbf{c}_t)\right]+{\mu M}/{2}\nonumber\\
	  \stackrel{(e)}{\leq}&{\Psi}^{*}+\mathbb{E}\left[(\mathbf{b}-\bm{\lambda}_t)^{\top}(\mathbf{A}\mathbf{x}_t+\mathbf{c}_t)\right]+{\mu M}/{2}
\end{align}
where (d) follows from the definition \eqref{eq.dual-func}, while (e) uses the weak duality and the fact that $\tilde{\Psi}^{*}\leq {\Psi}^{*}$.

Taking expectations on both sides of \eqref{eq.57} over $\mathbf{q}_t$, summing both sides over $t=1,\ldots,T$, dividing by $T$, and letting $T\rightarrow \infty$, we arrive at
\begin{align}\label{eq.43}
	\!\!\!&\lim_{T\rightarrow \infty} ({1}/{T}) \textstyle\sum_{t=1}^{T}\mathbb{E}\left[\Psi_t(\mathbf{x}_t)\right]\nonumber\\
	\!\!\! \stackrel{(f)}{\!\leq\!}&\!{\Psi}^{*}\!\!+\!\!\lim_{T\rightarrow \infty} \!({1}/{T})\!{\textstyle\sum_{t=1}^{T}}\!\mathbb{E}\!\left[(\mathbf{b}\!-\!\bm{\lambda}_t)\!^{\top}\!(\mathbf{A}\mathbf{x}_t\!+\!\mathbf{c}_t)\right]\!\!+\!\!\frac{\mu M}{2}\!+\!\!\lim_{T\rightarrow \infty}\!\!\!\frac{\mu\|\mathbf{q}_{1}\|^2}{2T}\nonumber\\
	  \!\!\!\stackrel{(g)}{\leq} &{\Psi}^{*}\!+\!\lim_{T\rightarrow \infty} ({1}/{T}) {\textstyle\sum_{t=1}^{T}}\mathbb{E}\left[(\mathbf{b}\!-\!\bm{\lambda}_t)^{\top}\!(\mathbf{A}\mathbf{x}_t+\mathbf{c}_t)\right]\!+\!\frac{\mu M}{2}\!
\end{align}
where (f) comes from $\mathbb{E}[\|\mathbf{q}_{T+1}\|^2]\geq 0$, and (g) follows because $\|\mathbf{q}_{1}\|$ is bounded.

What remains to show is that the second term in the RHS of \eqref{eq.43} is ${\cal O}(\mu)$.
Since $\bm{\lambda}_t$ converges to $\bm{\lambda}^*\; {\rm w.h.p.}$, there always exists a finite  $T_{\rho}$ such that for $t>T_{\rho}$, we have $\|\bm{\lambda}^*-\mathbf{b}-(\bm{\lambda}_t-\mathbf{b})\|\leq \rho,\;{\rm w.h.p.}$, and thus
\begin{align}\label{eq.44}
	&\lim_{T\rightarrow \infty} ({1}/{T}) {\textstyle\sum_{t=1}^{T}}\mathbb{E}\left[(\mathbf{b}-\bm{\lambda}_t)^{\top}(\mathbf{A}\mathbf{x}_t+\mathbf{c}_t)\right]\nonumber\\
=&\lim_{T\rightarrow \infty} 	({1}/{T}) {\textstyle\sum_{t=1}^{T}}\mathbb{E}\Big[\Big(\mathbf{b}\!-\!\bm{\lambda}^*\!+\!(\bm{\lambda}^*\!-\!\mathbf{b}\!-\!\bm{\lambda}_t\!+\!\mathbf{b})\Big)^{\!\!\top}\!\!(\mathbf{A}\mathbf{x}_t+\mathbf{c}_t)\Big]\nonumber\\
\stackrel{(h)}{\leq}&\lim_{T\rightarrow \infty}\! ({1}/{T}) {\textstyle\sum_{t=1}^{T}}\mathbb{E}\left[(\bm{\lambda}^*-\mathbf{b})^{\top}(-\mathbf{A}\mathbf{x}_t-\mathbf{c}_t)\right]+{\cal O}(\rho)\nonumber\\
\stackrel{(i)}{\leq}&\|\bm{\lambda}^*-\mathbf{b}\|\cdot \Big\|\lim_{T\rightarrow \infty} ({1}/{T}) {\textstyle\sum_{t=1}^{T}}\mathbb{E}\left[\mathbf{A}\mathbf{x}_t+\mathbf{c}_t\right]\Big\|+{\cal O}(\rho)
\end{align}
where (h) holds since $T_{\rho}<\infty$ and $\|\mathbf{A}\mathbf{x}_t+\mathbf{c}_t\|$ is bounded, and (i) follows from the Cauchy-Schwarz inequality.
Note that $\mathbf{A}\mathbf{x}_t+\mathbf{c}_t$ can be regarded as the service residual per slot $t$.
Since the steady-state queue lengths exist according to Theorem \ref{the.queue-stable},
it follows readily that there exist constants $D_1\!=\!\Theta(1/\mu)$, $D_2\!=\!\Theta(\sqrt{\mu})$ and $\tilde{B}\!=\!\Theta(1/\sqrt{\mu})$, so that
\begin{align}
	\mathbf{0}&\stackrel{(j)}{\leq} \lim_{T\rightarrow \infty} \frac{1}{T} \sum_{t=1}^{T}\mathbb{E}[-\mathbf{A}\mathbf{x}_t-\mathbf{c}_t]\stackrel{(k)}{\leq}\mathbf{1}\cdot\sqrt{M}\,\mathbb{P}\left(\mathbf{q}_t<\mathbf{1}\cdot\sqrt{M}\right)\nonumber\\
	&\stackrel{(l)}{\leq}\mathbf{1}\cdot\sqrt{M}D_1 e^{-D_2(\mathbf{b}/\mu-\tilde{B}-\sqrt{M})}
\end{align}
where (j) holds because all queues are stable (cf. Theorem \ref{the.queue-stable}); (k) follows since the maximum queue variation is $\sqrt{M}\!:=\!\max_t\max_{\mathbf{x}_t\in{\cal X}}\!\|\mathbf{A}\mathbf{x}_t+\mathbf{c}_t\|$, and negative accumulated service residual may happen only when the queue length $\mathbf{q}_t\!<\!\sqrt{M}$; and (l) comes from the large deviation bound of $\mathbf{q}_t$ in \cite[Lemma 4]{huang2014} and \cite[Theorem 4]{huang2011} with the steady-state queue lengths $\mathbf{b}/\mu$ in Theorem \ref{the.queue-stable}.
Setting $\mathbf{b}=\sqrt{\mu}\log^2(\mu)\cdot \mathbf{1}$, there exists a sufficiently small $\mu$ such that $-D_2(\frac{1}{\sqrt{\mu}}\log^2(\mu)-\tilde{B}-\sqrt{M})\leq 2\log(\mu)$, which implies that $\big\|\!\lim_{T\rightarrow \infty}\! \frac{1}{T}\! \sum_{t=1}^{T}\!\mathbb{E}[-\!\mathbf{A}\mathbf{x}_t\!-\!\mathbf{c}_t]\big\|\leq \|\mathbf{1}\|\cdot\sqrt{M}D_1\mu^2={\cal O}(\mu)$. Setting $\rho=\mathbf{o}(\mu)$, we have [cf. \eqref{eq.44}]
\begin{align}
	\lim_{T\rightarrow \infty} \frac{1}{T} \sum_{t=1}^{T}\mathbb{E}\left[(\mathbf{b}-\bm{\lambda}_t)^{\top}(\mathbf{A}\mathbf{x}_t+\mathbf{c}_t)\right]={\cal O}(\mu)
\end{align}
and the proof is complete.

\end{document}